\documentclass{amsart}

\usepackage{graphicx}%
\usepackage{multirow}%
\usepackage{amsmath,amssymb,amsfonts}%
\usepackage{amsthm}%
\usepackage{mathrsfs}%
\usepackage[title]{appendix}%
\usepackage{xcolor}%
\usepackage{textcomp}%
\usepackage{manyfoot}%
\usepackage{algorithm}%
\usepackage{algorithmicx}%
\usepackage{algpseudocode}%
\usepackage{listings}%
\usepackage{longtable}
\usepackage{tabularx}
\usepackage{booktabs}
\usepackage{tikz-cd}

\usepackage{natbib}
\newtheorem{theorem}{Theorem}[section]
\newtheorem{proposition}[theorem]{Proposition}
\newtheorem{lemma}[theorem]{Lemma}
\newtheorem{corollary}[theorem]{Corollary}

\theoremstyle{definition}

\newtheorem{definition}{Definition}[section]
\newtheorem{example}{Example}[section]

\theoremstyle{remark}
\newtheorem{remark}[definition]{Remark}

\newcommand{\forkindep}[1][]{%
	\mathrel{
		\mathop{
			\vcenter{
				\hbox{\oalign{\noalign{\kern-.3ex}\hfil$\vert$\hfil\cr
						\noalign{\kern-.7ex}
						$\smile$\cr\noalign{\kern-.3ex}}}
			}
		}\displaylimits_{#1}
	}
}
\newcommand{\dotminus}{\mathbin{\dot{-}}}

\numberwithin{equation}{section}

\begin{document}
	
	\author[Haoming Wang]{Haoming Wang}
	\title{Continuous First-Order Logic of Abstract Harmonic Spaces}
	\address{Center for Combinatorics, Nankai University, Tientsin 300071, China}
	\email{wanghm37@nankai.edu.cn}
	\urladdr{https://blueairm.github.io/}
	
	\subjclass[2020]{Primary 03C66; Secondary 31D05, 03C10.}
	\keywords{Martin compactification, Harmonic measure, Keisler measure, Riemann surface}
	
	\begin{abstract}
		%This paper studies the Brelot theory of abstract harmonic spaces and its language expansions within an unbounded many-sorted continuous first-order syntax. The results are divided into four parts.  (1) This theory together with finite interpolation schemes is shown to eliminate harmonic-variable quantifiers on finite evaluation fragments. On regular domains, the classical Dirichlet solution operator gives a definable Skolem function, allowing the auxiliary boundary sorts to be eliminated. (2) On a Greenian domain, after fixing a reference point, the Martin compactification is homeomorphic to a compact subspace of the local type space, and the Martin boundary corresponds precisely to non-principal Martin types. Consequently, there exists a Keisler measure supported on the minimal Martin types with complete lifting.  (3) Harnack compactness makes normalized local evaluation formulas stable. It therefore yields canonical bases so that the harmonic barycenter associated with a complete lift depends only on its underlying local measure. (4) First-jet expansions provide continuous logic formulations of Rad\'o--Kneser--Choquet and Lewy rigidity phenomena, and of thinness, polarity, and fine openness. A forking chain of Lebesgue thorn is constructed by this means.
		This paper studies Brelot harmonic spaces within an unbounded
		many-sorted continuous first-order language and its expansions.
		
		\begin{enumerate}
			\item In the finite evaluation reduct of the harmonic language, the Brelot theory together with the finite interpolation scheme eliminates bounded harmonic quantifiers while leaving point quantifiers untouched. 
			
			%Two further results develop model-theoretic counterparts of classical potential theory. 
			\item The Martin compactification is homeomorphic to a compact subset of the local type space, with the Martin boundary corresponding to non-principal types. 
			
			\item If the compact convex space of normalized positive harmonic functions is a Choquet simplex, then each Martin representing measure on the minimal boundary corresponds to a unique Keisler measure supported on minimal types and admits a complete lift.
			
			\item Harnack compactness also implies stability of normalized evaluation formulas, yielding canonical bases and a harmonic barycenter factoring through Keisler measures.
		\end{enumerate}
		
		These constructions are further applied to fine potential theory, where boundary poles of potentials are Keisler null. First-jet expansions encode the Rad\'o--Kneser--Choquet and Lewy planar rigidity phenomena, while higher-dimensional failures are analyzed through thorn-forking chains in o-minimal expansions.
		%compactness implies that normalized evaluation formulas are stable, yielding canonical bases and a harmonic barycenter factoring through local Keisler measures. Applying these results to fine potential theory, polarity is identified with Keisler measure zero. This paper also tries to encode Rad\'o--Kneser--Choquet and Lewy planar rigidity phenomena, and other higher-dimensional failure examples are finally analyzed by dividing and forking chains.
	\end{abstract}

	\maketitle
	
	\section{Introduction}
	
	Axiomatic potential theory, initiated by M. Brelot \cite{Brelot1960}, organizes harmonic functions on a locally compact, locally connected Hausdorff space. At its heart lies the interaction among harmonic sheaves, superharmonic functions, and balayage (sweeping) operators. Further developed by Constantinescu and Cornea \cite{CC1963axiomatic} and Bauer \cite{bauer2006harmonische}, this theory provides an axiomatic setting for the Dirichlet problem and Martin-boundary theory.
	
	The Martin representation theorem \cite{Martin1941harmonic} associates with a Greenian domain $U$ and a reference point $x_0\in U$ a minimal Martin boundary $\partial_m U$ whose points parameterize normalized minimal positive harmonic functions. Under the Choquet-simplex assumption \cite{Choquet1954capacity}, every positive harmonic function $h$ on $U$ admits a unique representation by a positive Borel measure $\mu_h$ on $\partial_m U$. By comparison, if $U$ is a regular domain and $g\in C(\partial U)$, the solution of the classical Dirichlet problem is represented by harmonic measure
	\begin{equation}
		h(x)
		=
		\int_{\partial U}g(y)\omega^x(\mathrm{d}y),
		\quad x\in U.
		\label{eq:Dirichlet-problem}
	\end{equation}
	Here $g$ is the prescribed continuous boundary datum, and $h$
	extends continuously to $\partial U$ with boundary values $g$. When the minimal Martin boundary can be identified with the topological boundary and the harmonic measures with poles $x$ and $x_0$ are mutually absolutely continuous,
	%% equivalent to connectedness
	the normalized Martin kernel is given by
	$K(x,y)=\mathrm{d}\omega^x/\mathrm{d}\omega^{x_0}(y)$
	for $\omega^{x_0}$-almost all $y\in\partial_m U$. In this case, the Martin
	representing measure of the positive Dirichlet solution is
	$\mu_h(\mathrm{d}y)=g(y)\omega^{x_0}(\mathrm{d}y)$.

	The analytic depth of Brelot harmonic spaces is characterized by the sheaf axiom, regular basis axiom, and the monotone convergence axiom scheme, the latter of which was shown to be equivalent to Harnack inequalities by Loeb and Walsh \cite{LW1965axiomatic}. The subsequent development of fine topology and boundary limits by El Kadiri and Fuglede \cite{ElKadiriFuglede2025} further connects excessive (superharmonic) functions to Markov processes \cite{Dynkin1962finetopo,doob1984classical}. Nevertheless, the classical treatment of balayage and polar sets remains tethered to heavy measure-theoretic machinery, most notably the capacity introduced by Choquet \cite{Choquet1954capacity}.
	
	Modern continuous first-order logic provides a natural syntax for
	metric structures carrying uniformly continuous functions and
	predicates. Its applications include group actions,
	operator algebras, and other analytic structures
	\cite{Berenstein2024freegroup,Dabrowski2019vonNeumann,GoldbringHoudayer2022wstar,Henson2017Gurarij}. The unbounded version
	is particularly well suited to harmonic function spaces equipped with
	compact-open metrics and gauges \cite{BenYaacov2008Unbounded}. Recent work by Miller
	\cite{Miller2026blms} and by Borgard and Miller \cite{Borgard_Miller_2026} concerns harmonic functions definable in o-minimal expansions of the real field. Rather than fixing one tame real structure, we treat the harmonic sheaf itself as a family of metric sorts and ask which classical constructions can be identified with local types, definable predicates, or measures.
	
	%To date, continuous model theory has emerged as a comprehensive language for metric structures, starting from Chang and Keisler \cite{ChangandKeisler1968continuousmodeltheory}, and developed by Ben Yaacov, Usvyatsov, Pedersen, Hanson \cite{Yaacov2007dfinite,Yaacov2008CONTINUOUSFO,Yaacov2010completeness,Hanson2025StronglyMinimalContinuousLogic} and others. In this setting, classical binary truth values $\{0,1\}$ are replaced by the unit $[0,1]$ or possibly an unbounded interval. Its applications now range from group actions \cite{Berenstein2024freegroup}, operator algebras \cite{Dabrowski2019vonNeumann,GoldbringHoudayer2022wstar} to many others \cite{Henson2017Gurarij}. However, the abstract harmonic space itself, despite providing canonical examples of metric structures, has remained relatively unexplored. Miller \cite{Miller2026blms} recently studied definable harmonic functions in an o-minimal expansion, but the abstract theory introduced here attempts to understand harmonicity in a different many-sorted formalism. Precisely, it asks which classical constructions identify with local types, definable predicates, or measures.

	In Section 2 and throughout this article, we fix a minimal harmonic language $\mathcal L_{\mathrm{Har}}$ in which the point space is a single sort, while regular domains form an external partially ordered set indexing the harmonic and boundary sheaf sorts. %Two nested theories and a scheme are mainly considered here: %the theory of Brelot harmonic space (BHS), the theory of non-trivial connected Brelot harmonic space (CBHS), and the finite interpolation scheme (FI). %and the theory of $n$-dimensional connected Brelot harmonic space (CBHS$_n$), $n\ge 2$. 
	The theory BHS of Brelot harmonic spaces consists of the sheaf, regular-basis, monotone-convergence, and compatibility axioms. Thus BHS describes pure Brelot structures, including Euclidean spaces, Riemann surfaces, and graphs. %By adding connectedness, CBHS rules out all models of discrete simple graphs. %The additional dimensionality assumption $n\ge2$ on CBHS$_n$ forces all infinite models of graphs to be impossible. %However, we may confront some tricky problems because a lot of familiar topological properties are unfortunately not first-order sentences in a single-sorted language. Thus, it urges a many-sorted language where a basis of open sets can be used for indexing sorts. This might require a general approach to axiomatizing harmonic function spaces, but we still choose it as our starting point. 
	The finite-interpolation scheme FI is an auxiliary axiom scheme expressing the abundance of harmonic functions. In conjunction with BHS, it yields relative elimination of harmonic quantifiers on finite-evaluation fragments: bounded harmonic quantifiers are eliminated, while point quantifiers remain untouched. For example, Euclidean spaces of dimension $n\ge2$ provide a broad class of harmonic functions satisfying finite interpolation.
	
	For analytic spaces, expanded languages usually have greater expressive power. We add two useful collections of symbols, Bal and Mar, to the minimal language. The Bal expansion introduced in Section 3 names balayage as a transfer-sort function, so the Dirichlet solution operator becomes a definable Skolem function. The Mar predicates introduced in Section 4 are designed for Greenian domains. In the resulting expansion, the Martin compactification is homeomorphic to a compact subset of local types determined by the compact-open closure of Cauchy nets, and the Martin boundary corresponds exactly to the non-principal types. Under the Choquet-simplex assumption, every Martin representing measure on the minimal boundary is identified with a unique Keisler measure supported on minimal types. Such a local measure admits complete lifts, whose collection is a nonempty compact convex subset of the space of probability measures. We call these correspondences the first and second isomorphism theorems.
	
	The results in Section 5 concern local
	stability. Harnack compactness implies stability of normalized evaluation formulas on compact subsets. The gluing technique for stable formulas by Ben Yaacov and Usvyatsov \cite{Yaacov2008CONTINUOUSFO} then yields an
	independence calculus for the resulting fixed stable family. Its
	canonical bases admit harmonic barycenters, and the barycenter
	associated with a complete Keisler lift depends only on the underlying local measure.

	A further smooth first-jet expansion is considered in Section 6 and applied to the Rad\'o--Kneser--Choquet and Lewy planar rigidity phenomena, together with their limitations in dimensions $n\ge3$. Boundary poles of potentials are shown to be null for the
	corresponding Keisler measures. Wood's counterexample and the classical Lebesgue thorn illustrate failures of rigidity and boundary regularity whose form is particularly sensitive to dimension. The latter also provides a concrete test case for the associated forking-independence phenomena.

	%This is because predicates are continuous real-valued functions in continuous first-order logic. This place needs a bit of caution. In classical first-order theory, values assigned to each formula share a common bound. Intuitively, a continuous function may be unbounded, while we have noticed that \cite{BenYaacov2008Unbounded} already deals with unbounded metric structures.  

	\section{Abstract harmonic spaces}
	In this section, we present the abstract harmonic spaces introduced by Brelot and formalize their theory in continuous first-order logic.
	
	\subsection{Abstract harmonic spaces}
	Let $(X,\tau)$ be a nonempty, locally compact, locally connected Hausdorff space.
	\begin{definition}[Presheaf]
		A presheaf $E=(E(U),\rho_{V,U})$ on $(X,\tau)$
		consists of the following data:
		\begin{itemize}
			\item for every open set $U\in\tau$, a set $E(U)$, and
			\item for every pair of open sets $V\subseteq U$, a restriction map
			$\rho_{V,U}:E(U)\to E(V)$.
		\end{itemize}
		These data are required to satisfy the following two conditions
		\begin{enumerate}
			\item (Identity) for every open set $U\in\tau$,
			$\rho_{U,U}=\operatorname{id}_{E(U)}$, and
			\item (Morphism) whenever $W\subseteq V\subseteq U$ are open sets,
			$\rho_{W,U}=\rho_{W,V}\circ\rho_{V,U}$.
		\end{enumerate}
		If $f\in E(U)$ and $V\subseteq U$, we write $\rho_{V,U}(f)=f|_V$.
	\end{definition}
	
	\begin{definition}[Sheaf]
		A presheaf $E=(E(U),\rho_{V,U})$ is a sheaf if it satisfies two additional
		conditions:
		\begin{enumerate}
			\item[(3)] (Locality) if $f,g\in E(\cup_{\alpha} U_{\alpha})$ agree on each open set $U_\alpha$, then $f=g$ on $\cup_{\alpha} U_{\alpha}$, and
			
			\item[(4)] (Gluing) if $f_\alpha\in E(U_\alpha)$ and the sections agree on every overlap $f_\alpha|_{U_\alpha\cap U_\beta}=f_\beta|_{U_\alpha\cap U_\beta}$, then there is a unique $f\in E(\cup_{\alpha}U_{\alpha})$ with $f|_{U_\alpha}=f_\alpha$ for every $\alpha$.
		\end{enumerate}
	\end{definition}
	
	%\begin{remark}   The full sheaf axiom is an analytic condition on the intended harmonic structures and quantifies over arbitrary open covers, so it is not a first-order sentence. Its finite part is nevertheless first-order. For every named finite cover, agreement of finitely many sections on pairwise overlaps and their gluing to a section can be expressed by a first-order scheme.\end{remark}

	For example, continuous functions on $(X,\tau)$ with the usual restriction maps form a sheaf.

	Let $U\in\tau$. Recall that an ultrafilter $\mathscr U$ on $U$ converges to $y\in\partial U$ if, for every open neighborhood $V$ of $y$, the intersection $V\cap U$ belongs to $\mathscr U$. A set $K$ is relatively compact in $U$ if its closure $\overline K$ is compact and contained in $U$, written $K\Subset U$. If $K\Subset X$, we simply call $K$ relatively compact. Suppose that $H=(H(U),\rho_{V,U})$ is a sheaf whose elements $h\in H(U)$ are called harmonic functions on $U$.
	
	\begin{definition}[Regular basis]
		A relatively compact open set $U$ is called regular if for every continuous function $g$ on $\partial U$ there exists a unique harmonic function $h$ on $U$ such that for every $y\in\partial U$ and every ultrafilter $\mathscr{U}$ on $U$ that converges to $y$ with respect to $\tau$, one has
		$$
		\lim_{\mathscr{U}\to y} h = g(y),
		$$
		and moreover $h\ge0$ whenever $g\ge0$. The function $h$ corresponding to $g$ is called the Dirichlet solution. A regular basis $\tau_0$ is a family of relatively compact regular open sets that forms a base for $\tau$.
	\end{definition}
	
	Since $X$ is locally compact and Hausdorff, every
	$g\in C(\partial U)$ admits, by the Tietze extension theorem, a
	bounded continuous extension to the compact closure of some relatively compact neighborhood of $\overline U$. Such an extension is generally not unique. Regularity of $U$, however, guarantees the existence of a unique continuous harmonic extension
	$h=D_U(g)\in H(U)\cap C(\overline U)$.
	
	We illustrate some examples where a regular basis exists.

	\begin{example}[Euclidean domain]\label{example1}
		Let $\Omega\subseteq\mathbb R^n$ be a connected open set with the classical Laplacian $\Delta$. The relatively compact Euclidean balls $B(a,r)\Subset\Omega$ form a regular basis because the Dirichlet problem is solvable on every ball.
		A countable basis is obtained by restricting to rational centers and rational radii.
	\end{example}

	\begin{example}[Riemann surface]
		Let $\mathcal R$ be a Riemann surface. Around every point $x\in\mathcal R$, choose a
		holomorphic coordinate chart
		$z_\alpha:U_\alpha\to V_\alpha\subseteq\mathbb C$.
		A regular basis is obtained by taking coordinate discs
		$\Omega=z_\alpha^{-1}(B(a,r))\Subset U_\alpha$.
		Let $\mathrm{d}s_\alpha$ be the arc-length element induced by $z_\alpha$. The Poisson kernel on $\Omega$ is
		$$P_{\Omega,\alpha}(x,\xi)
		=
		\frac{r^2-|z_\alpha(x)-a|^2}
		{2\pi r|z_\alpha(x)-z_\alpha(\xi)|^2}.$$
		Thus, %for every boundary function $f\in C(\partial\Omega)$, 
		the Dirichlet solution is given by the Poisson integral formula.
		%$$ h(x) = \int_{\partial\Omega} P_{\Omega,\alpha}(x,\xi)f(\xi)ds_\Omega(\xi), \quad x\in\Omega.$$ %where $ds_\Omega$ denotes the pullback of Euclidean arclength measure on $\partial B(a,r)$.
	\end{example}
	
	\begin{example}[Semialgebraic sets]
		Let $\Omega\subseteq\mathbb R^n$ be an open connected set definable in an o-minimal expansion of the real field. For the classical Laplacian, the balls $B(a,r)$ with $\overline{B(a,r)}\subseteq\Omega$ form a definable regular basis. Assume in addition that $\Omega$ is Greenian, and let $G(x,y)$ be its Green kernel off the diagonal. Given an exhaustion $U_0\Subset U_1\Subset\cdots$ of $\Omega$ by relatively compact regular domains, one metric for the compact-open topology of bounded Martin profiles is
		$$d_M(y,z)=\sum_{m<\omega}2^{-m-1}\sup_{x\in\overline U_m}|\widehat K(x,y)-\widehat K(x,z)|,$$
		where $K(x,y)=G(x,y)/G(x_0,y)$ and $\widehat K=K/(1+K)$. The Martin compactification is the completion of this profile space, after quotienting if $d_M$ is only a pseudometric. When $\Omega$ is regular, the Riesz representation theorem gives harmonic measure $\omega^x$ from the positive functional $g\mapsto D_\Omega(g)(x)$. If the topological and minimal Martin boundaries agree and the relevant harmonic measures are mutually absolutely continuous, then $K(x,\xi)=\mathrm{d}\omega^x/\mathrm{d}\omega^{x_0}(\xi)$ for $\omega^{x_0}$-almost every $\xi\in\partial\Omega$. Consequently, for $g\in C(\partial\Omega)$, the Dirichlet solution has the Martin representation
		$$
		\begin{aligned}
			h(x)
			=
			\int_{\partial\Omega}K(x,\xi)\mathrm{d}\mu_h(\xi)
		\end{aligned}
		$$
		where $\mathrm{d}\mu_h(\xi)=g(\xi)\mathrm{d}\omega^{x_0}(\xi)$. Without a simplex assumption, this measure is generally not unique.
	\end{example}

	\begin{example}[Discrete graph]\label{example: discrete graph}
		Let $G=(V,E)$ be a connected locally finite graph. Equip $V$ with the discrete topology and consider the graph Laplacian
		$$\Delta f(x)=f(x)-\frac{1}{\deg(x)}\sum_{y\sim x}f(y),$$
		where $y\sim x$ means that $y$ is adjacent to $x$. Thus the equation $\Delta f(x)=0$ is precisely the mean-value property. More generally, for an irreducible Markov chain $P=(p(x,y))$ supported on the edges, put $\Delta_P=I-P$, where $Pf(x)=\sum_{y\in V}p(x,y)f(y)$. For a finite nonempty set $U\Subset V$, define its outer graph
		boundary $$\partial_G U=\left\{
		\xi\in V\setminus U:
		\exists x\in U, p(x,\xi)>0
		\right\}.$$ Finite connected subgraphs $U\Subset V$, together with their outer boundaries $\partial_GU$, form the natural regular domains for discrete potential theory.
		
		Let $X_n$ be the corresponding Markov chain, and define the first exit time from $U$ by $\tau_U=\inf\{n\ge0:X_n\notin U\}.$ Since $U$ is finite and the chain is irreducible, one has
		$\mathbb P_x(\tau_U<\infty)=1$ for $x\in U$.
		The discrete Poisson kernel is the exit distribution
		$$P_U(x,\xi)
		=
		\mathbb P_x(X_{\tau_U}=\xi),
		\quad
		x\in U,\xi\in\partial_GU.$$ For every boundary function
		$f\in\mathbb R^{\partial_GU}$, the unique solution $u\in\mathbb R^U$ of the discrete Dirichlet problem satisfies $(I-Q)u=Rf$, where $Q=(p(x,y))_{x,y\in U}$ and $R=(p(x,\xi))_{x\in U,\,\xi\in\partial_GU}$. Thus $u$ is definable from the transition probabilities and boundary data by a finite linear system.
	\end{example}
	
	Henceforth, we assume that the nonempty, locally compact, locally connected Hausdorff space $X$ admits a regular basis $\tau_0$ and that the harmonic sheaf $H=(H(U),\rho_{V,U})$ is considered only for $U,V\in\tau_0$. %For each $U  \in \tau_0$, assume $1 \in H(U)$. % and $H(U)$ is a subspace of the Banach space $C(\overline U)$ of continuous functions on $\overline U$.
	
	\begin{definition}[Monotone convergence principle]
		A harmonic sheaf is said to satisfy the monotone convergence principle if, for any harmonic functions $h_n$ with $h_n\le h_{n+1}$ on a connected set $U\in\tau_0$, there exists an upper envelope $h= %\lim\limits_{\mathscr U} h_n 
		\sup_n h_n$ taking values in $(-\infty,+\infty]$ that is either harmonic or identically $+\infty$ on $U$.
	\end{definition}

	\begin{remark}%The original convergence axiom of Brelot is stated for an arbitrary upward-directed family $\mathcal F\subseteq\mathcal H(U)$. The axiom implies its pointwise upper envelope $\sup_{h\in\mathcal F} h$ is either in $H(U)$ or identically equals $+\infty$. \cite{CC1963axiomatic} proved these two axioms equivalent. 
		The monotone-convergence principle is an axiom scheme. The related Harnack inequality states that for every connected open set $U$, every compact set
		$K\Subset U$, and every nonnegative harmonic function $h\in H(U)$, there is a constant $C_{K,U}\ge1$, independent of $h$, such that
		$$
		\sup_{x\in K}h(x)
		\le
		C_{K,U}\inf_{x\in K}h(x).
		$$
		Under the sheaf and regular basis axioms, this family of inequalities is equivalent to the monotone convergence principle \cite[Theorem 1]{LW1965axiomatic}.
	\end{remark}

	%\begin{lemma}[Tietze]   For every $U \in \tau_0$ and $g \in C(\partial U)$, there exists a relatively compact open neighborhood $W$ of $\overline U$ and a bounded continuous extension $\tilde g \in C(\overline{W})$ that agrees with $g$ on $\partial U$.\label{lem: Tietze}\end{lemma}\begin{proof} Since $\overline U$ is compact and $X$ is locally compact, there exists an open set $W$ such that $\overline U \subset W$ and $\overline{W}$ is compact. The subspace $\overline{W}$ is compact Hausdorff, hence normal. The set $\partial U$ is a closed subset of $\overline{W}$, so by the Tietze extension theorem, $g$ extends to a continuous function $\tilde g \in C(\overline{W})$. Clearly, $\tilde{g}$ is bounded and agrees with $g$ on $\partial U$ since $\overline W$ is a compact neighborhood of $\partial U$.\end{proof}

	\begin{theorem}[Dirichlet solution and trace]\label{thm: Dirichlet solution and trace} Assume the regular-basis axiom and let $U\in\tau_0$. Then $D_U:C(\partial U)\to C(\overline U)\cap H(U)$ is a bijection whose inverse is the trace operator $\operatorname{Tr}_U$. Conversely, if $\operatorname{Tr}_U:C(\overline U)\cap H(U)\to C(\partial U)$ is a surjective isometry, then the regular Dirichlet problem on $U$ is solvable and unique.
	\end{theorem}
	\begin{proof}
		For every $g\in C(\partial U)$, regularity gives a unique $h\in H(U)$ extending continuously to $\overline U$ with $h|_{\partial U}=g$. Thus $D_U(g)=h$ and $\operatorname{Tr}_U(D_Ug)=g$.
		
		Conversely, if $h\in C(\overline U)\cap H(U)$, uniqueness gives $D_U(\operatorname{Tr}_Uh)=h$. Hence the two maps are inverse. For $h_1,h_2\in C(\overline U)\cap H(U)$, the maximum principle gives
		$$
		\|h_1-h_2\|_{C(\overline U)}
		=
		\|\operatorname{Tr}_U(h_1)-\operatorname{Tr}_U(h_2)\|_{C(\partial U)}.
		$$
		Therefore $\operatorname{Tr}_U$ is an isometry. Conversely, surjectivity produces a continuous harmonic extension for every boundary datum, and injectivity gives uniqueness.
	\end{proof}

	\begin{definition}[Brelot harmonic space] A Brelot harmonic space is a harmonic sheaf which admits a regular basis and satisfies the monotone convergence principle.
	\end{definition}
	
	\subsection{Continuous first-order logic}
	
	%\begin{definition}The logical symbols of continuous first-order logic are\begin{itemize}    \item Parentheses: $( \ , \ )$    \item Connectives: $\dotminus$    \item Quantifiers: $\sup$ (or $\inf$)    \item Variables: $v_0, v_1, \dots$    \item Constants: $\mathbb Q \cap [0,1]$    \item A metric: $d$\end{itemize}\end{definition}
	
	We use the logical symbols of continuous first-order logic, including parentheses, the connectives $\dotminus$, $\neg$, and $\frac12$, the quantifiers $\sup$ and $\inf$, and variables $x_1,x_2,\ldots$, as in \cite{Yaacov2008CONTINUOUSFO}. Our terminology for unbounded continuous logic follows \cite{BenYaacov2008Unbounded}. %Extend the formal binary operation $\dotminus$ on terms $x\dotminus y = \max(x-y,0)$ and on well-formed formulas $v(\varphi\dotminus \psi)= \max(v(\varphi) - v(\psi), 0)$ for any evaluation map $v$ to the set of all formulas by structural induction. %The only difference between continuous first-order logic and binary first-order logic is that continuous first-order logic contains a real closed field as meta-theory, so we can easily deal with the closure operation and convergence of points.
	%. In the context of continuous logic, it is extended to all formulas as follows       \begin{itemize}        \item for terms $t_1,t_2$, this operation is defined pointwise as usual;        \item for atomic formulas $\varphi,\psi$, $v(\varphi\dotminus \psi)= \max(v(\varphi) - v(\psi), 0)$ for any evaluation map $v$;   \item for arbitrary well-formed formulas $\varphi, \psi$, the operation is extended recursively. Assuming the variable $x$ is not free in $\psi$, we have the following logical equivalences\end{itemize}              \begin{align*}            (\sup_x \varphi) \dotminus \psi &\equiv \sup_x (\varphi \dotminus \psi), \\            (\inf_x \varphi) \dotminus \psi &\equiv \inf_x (\varphi \dotminus \psi), \\   \psi \dotminus (\sup_x \varphi) &\equiv \inf_x (\psi \dotminus \varphi), \\            \psi \dotminus (\inf_x \varphi) &\equiv \sup_x (\psi \dotminus \varphi).        \end{align*}By structural induction, these rules ensure that $\dotminus$ can penetrate any quantifier prefix, making it a globally well-defined operation on the set of all formulas. 

	\begin{definition}[Language Har]
		%Fix a locally compact, locally connected Hausdorff base space $X$. Let $(\tau_0,\subseteq)$ be a partially ordered family of relatively compact regular open sets of $X$. %The elements$U,V\in\tau_0$ are external indices of the language and are not variables of the object language.
		The unbounded many-sorted continuous first-order language
		$\mathcal L_{\mathrm{Har}}$ first contains the following sort:
		
		\begin{itemize}
			
			\item (Point sort) $(X,d_X,\nu_X)$, a gauged metric space, where $d_X:X\times X\to[0,\infty)$ is a metric and $\nu_X:X\to[0,\infty)$ is a gauge.
			
		\end{itemize}
		Let $(\tau_0,\subseteq)$ be a partially ordered family of relatively compact regular open subsets of the base space $X$. For every $U\in\tau_0$, the language $\mathcal L_{\mathrm{Har}}$ also contains
		
		\begin{itemize}
			\item (Domain sort) $(X_U,d_U,\nu_U)$, a gauged metric space,
			
			\item (Harmonic sort)
			$(H(U),+_{H(U)},\cdot_{H(U)},d_{H(U)},\nu_{H(U)})$, a gauged metric linear space of harmonic functions on $U$,
			
			\item (Boundary sort) $(B(U),+_{B(U)},\cdot_{B(U)},d_{B(U)},\nu_{B(U)})$, a gauged metric linear space of boundary functions on $\partial U$,
			
			\item (Embedding function) $i_U:X_U\to X$,
			
			\item (Solution function) $D_U:B(U)\to H(U)$, with $D_U(f)=h$, a positive linear operator,
			
			and, for every inclusion $V\subseteq U$ in $\tau_0$,
			
			\item (Inclusion function) $j_{V,U}:X_V\to X_U$,
			
			\item (Restriction function) $\rho_{V,U}:H(U)\to H(V)$, with $\rho_{V,U}(h)=h|_V$, a linear embedding,
			
			\item (Interior predicate) 
			$\operatorname{Eval}_U:X_U\times H(U)
			\to
			\mathbb R$, with
			$\operatorname{Eval}_U(x,h)=h(x)$.
			%\item (Boundary predicate) ${\rm Eval}_{U,\xi}:B(U)    \rightarrow    \mathbb R,$ $    {\rm Eval}_{U,\xi}(f)=f(\xi)$.
		\end{itemize}
	\end{definition}
	
	\begin{definition}[Compatibility axiom] \label{def: Compatibility axiom}
		The predicates and functions are required to satisfy the following compatibility axioms. For every $U\in\tau_0$,
		\begin{enumerate}
			\item (Identity) $j_{U,U}=\operatorname{id}_{X_U}$.

            For every $V\subseteq U$, $x\in X_V$, and
			$h\in H(U)$,
			\item (Restriction) $
			\operatorname{Eval}_V\bigl(x,\rho_{V,U}(h)\bigr)
			=
			\operatorname{Eval}_U\bigl(j_{V,U}(x),h\bigr).
			$

            For every $W\subseteq V\subseteq U$,
			\item (Morphism) 
			$i_U\circ j_{V,U}=i_V$ and
			$j_{W,U}=j_{V,U}\circ j_{W,V}$.

            For every $U\in\tau_0$, $x\in X_U$,
			$h_1,h_2\in H(U)$, and named scalars $\lambda_1,\lambda_2$,
            
			\item (Linearity) 
			$\operatorname{Eval}_U(x,\lambda_1 h_1+\lambda_2 h_2)
			=
			\lambda_1\operatorname{Eval}_U(x,h_1)
			+
			\lambda_2\operatorname{Eval}_U(x,h_2).$
		\end{enumerate}
	\end{definition}

	%\begin{remark}   For the boundary sort $C(\partial U)$, one should be more careful since $\operatorname{Eval}_{\xi}:(\xi,f)\mapsto f(\xi)$ is not uniformly continuous in $\xi$ for the entire space $C(\partial U)$. Thus, it is safer to include only boundary point evaluation maps $\operatorname{Eval}_{\xi}$ for named boundary points $\xi$ when needed.\end{remark}
	
	\begin{remark} Throughout the rest of this paper, $d_X$ is assumed to be compatible with the topology $\tau$. The boundary sort $B(U)$ is intended to represent $C(\partial U)$. For every $U\in\tau_0$, regularity of $U$ means that each boundary function $f\in C(\partial U)$ has a unique harmonic extension $h\in H(U)$. Model-theoretically, the Dirichlet solution operator $D_U:C(\partial U)\to H(U)$ is a Skolem function. Theorem~\ref{thm: Dirichlet solution and trace} shows that $D_U$ is $1$-Lipschitz in the supremum norm and hence is a legitimate function symbol in continuous first-order logic. \end{remark}
	
	%\begin{remark}For every $U \in \tau_0$, writing $  \delta_U(x) :=d_X(\iota_U(x),X\setminus U),$ one may take $\nu_U(x)    :=    {1}/{\delta_U(x)}$ and $  d_U(x,y):=    d_X(\iota_U(x),\iota_U(y))    +    |\nu_U(x)-\nu_U(y)|.$ Thus points approaching $\partial U$ have gauge tending to $+\infty$, while every gauge bounded part of $X_U$ lies in a compact subset of $U$. Choose a compact exhausting $K_0\Subset K_1 \Subset \dots \Subset U$. Thus, there exists a natural metric on $H(U)$ induced by the family of semi-norms $p_m(h)=\sup_{x\in K_m} |h(x)|$$d_{H(U)}(h_1,h_2) = \sum_{m=1}^{\infty}2^{-m} \left[p_m(h_1 - h_2) \land 1\right],$where the choice of a compact exhaustion is independent of this Fr\'echet metric.\end{remark}

	Fix a continuous many-sorted language $\mathcal L$. For example, take $\mathcal L=\mathcal L_{\mathrm{Har}}$.
	\begin{definition}[Structure]
		An $\mathcal L$-structure, or model, $\mathfrak M=(M_1,M_2,\ldots;\sigma)$ consists of nonempty sets $M_1,M_2,\ldots$ and an interpretation function $\sigma$ that assigns a metric space to each sort and a continuous function with a prescribed uniform-continuity modulus to each predicate and function symbol.
	\end{definition}
	
	Recall that an $\mathcal L$-formula without free variables is an $\mathcal L$-sentence. An assignment $\chi$ in $\mathfrak M=(M_1,M_2,\dots;\sigma)$ sends each variable to the appropriate sort.
	
	\begin{lemma}[Coincidence lemma]
		\label{lem: coincidence}
		Let $\mathfrak M$ be an $\mathcal L$-structure, let $t$ be an $\mathcal L$-term, let $\varphi$ be an $\mathcal L$-formula, and let $\chi_1,\chi_2$ be assignments in $\mathfrak M$.
		
		\begin{enumerate}
			\item If $\chi_1(x)=\chi_2(x)$ for every variable $x$ occurring in $t$, then
			$t^{\mathfrak M,\chi_1}
			=
			t^{\mathfrak M,\chi_2}$.
			
			\item If $\chi_1(x)=\chi_2(x)$ for every variable $x$ occurring freely in $\varphi$, then
			$\varphi^{\mathfrak M,\chi_1}
			=
			\varphi^{\mathfrak M,\chi_2}$.
		\end{enumerate}
	\end{lemma}
	The coincidence lemma follows by induction on terms and formulas \cite[Section 1.4]{ChangandKeisler2013modeltheory}. In particular, if $t$ is a term with no variables, then its interpretation is independent of the assignment, and we write it simply as $t^{\mathfrak M}.$ If $\varphi$ is an $\mathcal L$-sentence, then its value is independent of the assignment, and we write $\varphi^{\mathfrak M}.$ %For an $\mathcal L$-formula $\varphi(\bar x_1, \bar x_2,\dots)$ with free variables $\bar x_1, \bar x_2,\dots$, the $\mathcal L$-sentence $\sup_{\bar x} \varphi$ is called its universal closure where only finitely many variables occur in $\sup_{\bar x} = \sup_{\bar x_1}\sup_{\bar x_2}\dots$. An $\mathcal L$-structure $\mathfrak M$ is called a model of an $\mathcal L$-formulas $\varphi$, denoted as $\mathfrak M\models \varphi$ if its universal closure of $\sup_{\bar x} \varphi$ assigns to 0 in $\mathfrak M$.
	
	\begin{definition}[Theory] An $\mathcal L$-theory $T$ is a set of $\mathcal L$-sentences. A structure $\mathfrak M$ satisfies an $\mathcal L$-sentence $\varphi$, written $\mathfrak M\models \varphi$, when $\varphi^{\mathfrak M} = 0$. $\mathfrak M$ is a model of $T$, written $\mathfrak M \models T$, when it satisfies every sentence in $T$. An $\mathcal L$-theory $T$ is called satisfiable if there is a model of $T$. The full theory $\operatorname{Th}(\mathfrak M)$ of an $\mathcal L$-structure $\mathfrak M$ consists of all sentences having value zero in $\mathfrak M$.
	\end{definition}
	
	Let $\mathfrak M,\mathfrak M'$ be two $\mathcal L$-structures. An elementary embedding is an injective map $\eta:M\to M'$ that preserves the interpretation of every symbol and satisfies $(\mathfrak M,\chi)\models\varphi$ if and only if $(\mathfrak M',\eta\circ\chi)\models\varphi$ for every $\mathcal L$-formula $\varphi(\bar x)$ and every assignment $\chi$ sending $\bar x$ to $\bar a$. If $M\subseteq M'$, we say that $\mathfrak M'$ is an elementary extension of $\mathfrak M$. Two models $\mathfrak M$ and $\mathfrak M'$ are elementarily equivalent if $\mathfrak M\models\varphi$ if and only if $\mathfrak M'\models\varphi$ for every sentence $\varphi$.
	A satisfiable $\mathcal L$-theory $T$ is called complete if any two models $\mathfrak M_1$ and $\mathfrak M_2$ of $T$ are elementarily equivalent. By construction, $\operatorname{Th}(\mathfrak M)$ is complete.
	
	%\begin{definition}[Lebesgue covering dimension]   Let $\mathcal U$ be a family of subsets of $X$. Define $\operatorname{mult}(\mathcal U)$ as the smallest ordinal $\alpha$ such that when every point of $X$ belongs to at most $\alpha+1$ members of $\mathcal U$. $\dim (X)$ is the smallest ordinal $\beta$ such that every open cover $\mathcal U$ of $X$ admits an open refinement $\mathcal U'$ satisfying $  \operatorname{mult}(\mathcal U')\le \beta+1.$\end{definition}

	\begin{definition}[BHS] The theory BHS consists of the $\mathcal L_{\mathrm{Har}}$-sentences for Brelot harmonic spaces. Its axioms are the compatibility and sheaf axioms, the regular-basis axiom, and the monotone-convergence scheme.
		%\begin{itemize}   \item  BHS = Compatibility + Sheaf + Regular basis + Monotone convergence. \item  CBHS = BHS + $X$ is connected and contains at least two points. 
		%   \item CBHS$_n$ = CBHS + $\dim (X) = n$.
		%\item BHS$^{\rm FI}$ = BHS + FI, CBHS$^{\rm FI}$ = CBHS + FI.\end{itemize}
\end{definition} 

%\begin{remark}[Sheaf axiom and finitary syntax]\label{rem:sheaf-syntax}The full sheaf axiom is an analytic condition on the intended Brelot structures and quantifies over arbitrary open covers, so it is not asserted to be finitary first-order sentences. Its finite part is nevertheless first-order. For every named finitecover, compatibility of finitely many sections on the pairwise overlaps and their gluing to a section on $U$ can be expressed by a bounded continuous first-order scheme, since only finitely many harmonic variables and restriction maps occur. We refer to these sentences as the finite sheaf scheme.\end{remark}

%\begin{remark}	Connectedness and dimension require separate coding. One possible connected expansion adds a sort of finite chains of basic domains and axioms saying that any two point parameters occur in a chain with successive nonempty overlaps. In a smooth $n$-dimensional expansion, one instead adds chart sorts and maps $\chi_U:X_U\to\mathbb R^n$ so that the coordinate arity fixes $n$, and transition maps satisfy the usual compatibility conditions. We refer to these coded expansions descriptively when an example needs them, but do not denote them by new global theories.\end{remark}

\begin{example}[Incompleteness of $\mathrm{BHS}$]
	\label{ex:BHS-incomplete}
	Example~\ref{example: discrete graph} gives finite and infinite graph models whose point sort carries the discrete metric. If $\mathfrak G_N$ has exactly $N$ points and $\mathfrak G_\infty$ has infinitely many, consider the sentence
	$$\sigma_N=\sup_{x_0,\ldots,x_N\in X}\min_{0\le i<j\le N}d_X(x_i,x_j).$$
	Every $(N+1)$-tuple in $\mathfrak G_N$ repeats a coordinate, so $\sigma_N^{\mathfrak G_N}=0$. In $\mathfrak G_\infty$ one can choose $N+1$ distinct points, so $\sigma_N^{\mathfrak G_\infty}=1$. Hence $\mathrm{BHS}$ is not complete.

	Example~\ref{example1} also yields non-elementarily-equivalent models. Assume that the language contains a distinguished regular-domain
	label $U_*$ and the bounded Dirichlet harmonic sort
	$$
	H^D_1(U_*)
	=
	\{h\in H^D(U_*):\nu_{H^D(U_*)}(h)\le 1\}.
	$$
	Fix a rational number $r$ with $0<r<1$. For every $N\ge2$,
	consider the sentence
	$$
	\sigma^D_{N,r}
	=
	\inf_{h_1,\ldots,h_N\in H^D_1(U_*)}
	\max_{1\le i<j\le N}
	\bigl(r-d_{H^D(U_*)}(h_i,h_j)\bigr)_+.
	$$
	Its value is zero precisely when the unit ball
	$H^D_1(U_*)$ contains an $N$-element $r$-separated subset.
	
	Let $M_1$ be the classical harmonic structure on
	$\mathbb R$, with
	$U_*^{M_1}=(-1,1)$.
	Every harmonic function on $(-1,1)$ is affine. The trace map
	$h\mapsto\bigl(h(-1),h(1)\bigr)$
	identifies $H^D(U_*^{M_1})$ isometrically with a
	two-dimensional normed space. Hence its closed unit ball is compact.
	It is therefore totally bounded, so there exists $N_0$ such that
	$H^D_1(U_*^{M_1})$ contains no $N_0$-element
	$r$-separated subset. Since the infimum in
	$\sigma^D_{N_0,r}$ is taken over a compact space, it follows that
	$(\sigma^D_{N_0,r})^{M_1}>0$.
	
	Let $M_2$ be the classical harmonic structure on
	$\mathbb R^2$, with
	$U_*^{M_2}=B(0,1)$.
	The Dirichlet operator is an isometric isomorphism
	$D_{U_*}:C(S^1)\to H^D(B(0,1))$.
	Thus $H^D(U_*^{M_2})$ is infinite-dimensional. By Riesz
	lemma, its closed unit ball contains arbitrarily large finite
	$r$-separated subsets. Consequently,
	$(\sigma^D_{N,r})^{M_2}=0$
	for every $N\ge2$. In particular,
	$(\sigma^D_{N_0,r})^{M_1}
	\neq
	(\sigma^D_{N_0,r})^{M_2}.$
	Hence the two connected harmonic structures are not elementarily
	equivalent.
\end{example}

%Similarly, one can show that CBHS$_n$ is also not complete.

Write $H^D(U)=C(\overline U)\cap H(U)$. Theorem~\ref{thm: Dirichlet solution and trace} shows that the two Skolem functions $C(\partial U)
\underset{\operatorname{Tr}_U}{\overset{D_U}{\rightleftarrows}}
H^D(U)$ are inverse isometries.

\begin{theorem}[Elimination of the auxiliary boundary sort]
	\label{thm:boundary-sort-elimination}
	For every formula $\varphi(g,\bar z)$ with $g\in B(U)$, there is a formula $\varphi^D(h,\bar z)$ with $h\in H^D(U)$ such that $\varphi(g,\bar z)=\varphi^D(D_U(g),\bar z)$.
	Consequently,
	$$\begin{aligned}
	    \inf_{g\in B(U)}\varphi(g,\bar z)&=\inf_{h\in H^D(U)}\varphi^D(h,\bar z),\\
        \sup_{g\in B(U)}\varphi(g,\bar z)& =\sup_{h\in H^D(U)}\varphi^D(h,\bar z).
	\end{aligned}$$
\end{theorem}
\begin{proof}
	By Theorem~\ref{thm: Dirichlet solution and trace}, $D_U:B(U)\to H^D(U)$ and $\operatorname{Tr}_U:H^D(U)\to B(U)$ are inverse definable isometries. Define $\varphi^D(h,\bar z)=\varphi(\operatorname{Tr}_U(h),\bar z)$. Then $\varphi^D(D_U(g),\bar z)=\varphi(g,\bar z)$. The translation commutes with continuous connectives. Since $D_U$ is bijective onto $H^D(U)$, it also commutes with $\inf$ and $\sup$ over the corresponding sorts. Induction on formulas proves the result.
\end{proof}

Although BHS is incomplete, adding FI yields a relative result. The next subsection proves relative quantifier elimination on the finite-evaluation fragment and the resulting relative model-completeness.

\subsection{Finite interpolation}
\label{subsec:finite-interpolation-relative-QE}

We work in the fixed language $\mathcal L_{\mathrm{Har}}$. Throughout this subsection, a tuple $\bar r\in\mathbb R^m$ is metalinguistic notation for numerical values of evaluation predicates. In a formula, its entries occur only as constants in continuous connectives or as slots into which real-valued formula outputs are substituted.

\begin{definition}[Finite interpolation]
	\label{def:finite-interpolation}
	Fix $U\in\tau_0$, let $\bar x=(x_1,\ldots,x_m)\in U^m$, and define the finite evaluation map by $\operatorname{ev}_{\bar x}(h)=(\operatorname{Eval}_U(x_1,h),\ldots,\operatorname{Eval}_U(x_m,h))$. The condition $\mathrm{FI}_{U,m}$ says that $\operatorname{ev}_{\bar x}:H(U)\to\mathbb R^m$ is surjective whenever the entries of $\bar x$ are pairwise distinct. The finite-interpolation scheme is $\mathrm{FI}=\{\mathrm{FI}_{U,m}:U\in\tau_0,\ m\ge1\}$.
\end{definition}

For a compact set $K\Subset U$ and $\delta>0$, put $\operatorname{sep}_U(\bar x)=\min_{i<j}d_U(x_i,x_j)$ and
$\Gamma_{K,m,\delta}=\{\bar x\in K^m:\operatorname{sep}_U(\bar x)\ge\delta\}.$
This is the separated compact configuration set. Semantically, FI says that an unrestricted harmonic quantifier ranges over every numerical tuple of values at each $\bar x\in\Gamma_{K,m,\delta}$. The next lemma supplies the uniform gauge bound needed for bounded quantifiers.

\begin{lemma}[Uniform finite interpolation]
	\label{lem:uniform-finite-interpolation}
	Assume FI. For every $K\Subset U$, $m\ge1$, $\delta>0$, and $S>0$, there is a constant $C_{U,K,m,\delta,S}$ such that, for every $\bar x\in\Gamma_{K,m,\delta}$ and every $\bar r\in[-S,S]^m$, some $h\in H(U)$ satisfies $\operatorname{ev}_{\bar x}(h)=\bar r$ and $\nu_{H(U)}(h)\le C_{U,K,m,\delta,S}$. The constants may be chosen so that
	$C_{U,K,m,\delta,S}=S C_{U,K,m,\delta,1}.$
\end{lemma}

\begin{proof}
	Fix $\bar x\in\Gamma_{K,m,\delta}$. By FI, there are $h_1^{\bar x},\ldots,h_m^{\bar x}\in H(U)$ with $h_i^{\bar x}(x_j)=1$ when $i=j$ and $h_i^{\bar x}(x_j)=0$ otherwise. For $\bar y$ near $\bar x$, let $M_{\bar x}(\bar y)$ be the matrix whose $(j,i)$-entry is $\operatorname{Eval}_U(y_j,h_i^{\bar x})$. This matrix depends continuously on $\bar y$ and equals the identity at $\bar x$. After shrinking to a neighborhood $N_{\bar x}$, we therefore have $\|M_{\bar x}(\bar y)^{-1}\|_{\infty\to\infty}\le2$ for $\bar y\in N_{\bar x}$.
	
	If $\|\bar r\|_\infty\le1$ and $\bar y\in N_{\bar x}$, set $\bar c=M_{\bar x}(\bar y)^{-1}\bar r$ and $h=\sum_{i=1}^m c_i h_i^{\bar x}$. Then $\operatorname{ev}_{\bar y}(h)=\bar r$ and
	$\nu_{H(U)}(h)\le2\sum_{i=1}^m\nu_{H(U)}(h_i^{\bar x}).$
	The compact set $\Gamma_{K,m,\delta}$ is covered by finitely many of the neighborhoods $N_{\bar x}$. Taking twice the largest corresponding sum gives $C_{U,K,m,\delta,1}$. Absolute homogeneity of the gauge then gives the stated value for general $S$.
\end{proof}

\begin{definition}[Minimal interpolation gauge]
	\label{def:minimal-interpolation-gauge}
	For a pairwise distinct tuple $\bar x\in U^m$ and a numerical target $\bar r\in\mathbb R^m$, define the minimal interpolation gauge
	$$q_{U,\bar x}(\bar r)=\inf\{\nu_{H(U)}(h):h\in H(U),\ \operatorname{ev}_{\bar x}(h)=\bar r\}.$$
	%This is the quotient gauge induced by the linear surjection $\operatorname{ev}_{\bar x}$. Equivalently, it is the Minkowski functional of the convex balanced set $\operatorname{ev}_{\bar x}(\{h:\nu_{H(U)}(h)\le1\})$. Thus $q_{U,\bar x}$ is determined by the finite evaluation map and the original harmonic gauge.
\end{definition}

\begin{theorem}[Uniform definability of the interpolation gauge]
	\label{thm:uniform-definability-interpolation-gauge}
	Assume FI, fix $K\Subset U$, $m\ge1$, $\delta>0$, and $S>0$, and put $C=C_{U,K,m,\delta,1}$. For every $\bar x\in\Gamma_{K,m,\delta}$ and $\bar r\in[-S,S]^m$, one has the exact penalty identity
	$$q_{U,\bar x}(\bar r)=\inf_{h\in H_{CS}(U)}\left(\nu_{H(U)}(h)+C\|\operatorname{ev}_{\bar x}(h)-\bar r\|_\infty\right),$$
	where $H_{CS}(U)=\{h\in H(U):\nu_{H(U)}(h)\le CS\}$. Moreover, $q_{U,\bar x}$ is $C$-Lipschitz in the target, the family is uniformly continuous in $\bar x$ on $\Gamma_{K,m,\delta}$, and these functionals form a uniformly definable predicate schema on the indicated bounded ranges.
\end{theorem}

\begin{proof}
	Lemma~\ref{lem:uniform-finite-interpolation}, followed by scaling, gives $q_{U,\bar x}(\bar a)\le C\|\bar a\|_\infty$ for every numerical tuple $\bar a$. Subadditivity of the quotient gauge therefore gives, for every $h\in H(U)$,
	$$q_{U,\bar x}(\bar r)\le\nu_{H(U)}(h)+C\|\operatorname{ev}_{\bar x}(h)-\bar r\|_\infty.$$
	Taking the infimum proves one inequality in the exact-penalty identity. Conversely, exact interpolants with gauges approaching $q_{U,\bar x}(\bar r)$ give the reverse inequality whenever $q_{U,\bar x}(\bar r)<CS$. If equality holds, the zero function already gives penalty at most $C\|\bar r\|_\infty\le CS=q_{U,\bar x}(\bar r)$. Hence the infimum may be restricted to the single gauge-bounded harmonic set $H_{CS}(U)$.
	
	For numerical targets $\bar r$ and $\bar s$, subadditivity and the preceding estimate yield $|q_{U,\bar x}(\bar r)-q_{U,\bar x}(\bar s)|\le C\|\bar r-\bar s\|_\infty$. In the exact-penalty formula, the quantified harmonic variable ranges over the fixed set $H_{CS}(U)$. Evaluation is uniformly continuous on $K\times H_{CS}(U)$, so taking an infimum preserves a common modulus in $\bar x$. This proves uniform continuity on the separated compact configuration set.
	
	Finally, for each fixed numerical $\bar r$, the right-hand side is a bounded $H(U)$-sort quantifier applied to the original predicates $\nu_{H(U)}$ and $\operatorname{Eval}_U$. The same expression permits uniformly continuous real-valued formula outputs to be substituted for the entries of $\bar r$. After the standard bounded rescaling of connectives, these expressions have common bounds and continuity moduli. The exact penalty identity therefore proves uniform definability.
\end{proof}

\begin{definition}[Bounded finite evaluation image]
	\label{def:bounded-finite-evaluation-image}
	For $R>0$ and a pairwise distinct tuple $\bar x\in U^m$, let $H_R(U)=\{h\in H(U):\nu_{H(U)}(h)\le R\}$ and define the bounded finite evaluation image by
	$E_{U,R}(\bar x)=\operatorname{ev}_{\bar x}(H_R(U))\subseteq\mathbb R^m.$
	%This notation describes an external finite-dimensional image; it does not name a new sort.
\end{definition}

A formula belongs to the finite evaluation fragment if each quantified harmonic variable occurs only through finitely many predicates $\operatorname{Eval}_U(x_i,h)$ and through the gauge bound on its quantifier. %Numerical slots in continuous connectives are not language variables.

\begin{theorem}[Bounded finite evaluation elimination]
	\label{thm:relative-finite-profile-QE}
	Assume FI and fix $K\Subset U$, $m\ge1$, $\delta>0$, and $R>0$. Uniformly for $\bar x\in\Gamma_{K,m,\delta}$,
	$$\{\bar r:q_{U,\bar x}(\bar r)<R\}\subseteq E_{U,R}(\bar x)\subseteq\{\bar r:q_{U,\bar x}(\bar r)\le R\},$$
	and the closure of $E_{U,R}(\bar x)$ is exactly $\{\bar r:q_{U,\bar x}(\bar r)\le R\}$. Consequently, every quantifier over $H_R(U)$ in the finite-evaluation fragment is uniformly eliminable on $\Gamma_{K,m,\delta}$ in the definitional presentation generated by the predicates $q_{U,\bar x}$. The resulting predicate is a uniform limit of finite continuous combinations of instances of $q$, point data, and the original evaluation data.
\end{theorem}

\begin{proof}
	If $\bar r\in E_{U,R}(\bar x)$, then $q_{U,\bar x}(\bar r)\le R$. If the last inequality is strict, the definition of the infimum supplies an exact interpolant of gauge less than $R$, giving the first inclusion. Continuity of $q$ shows that the closed sublevel set contains the closure of $E_{U,R}(\bar x)$. Conversely, if $q_{U,\bar x}(\bar r)=R$, choose exact interpolants $h_n$ with $\nu_{H(U)}(h_n)\le R+n^{-1}$. The functions $R(R+n^{-1})^{-1}h_n$ lie in $H_R(U)$ and their evaluation tuples converge to $\bar r$. This proves the closure identity without assuming that the infimum defining $q$ is attained.
	
	It remains to express the quantified value without numerical variables. Consider one innermost subformula on a fixed bounded locus,
	$$\Psi(\bar z,\bar x)=\inf_{h\in H_R(U)}\Phi\bigl(\bar z,\operatorname{ev}_{\bar x}(h)\bigr),$$
	where the second argument of $\Phi$ denotes $m$ real-valued formula slots. By the bounded locus convention there is $B=B_{U,K}>0$ such that $|\operatorname{Eval}_U(x,h)|\le B\nu_{H(U)}(h)$ for $x\in K$. Hence all relevant tuples lie in $[-BR,BR]^m$, and the closure identity gives the external equality
	$$\Psi(\bar z,\bar x)=\inf\{\Phi(\bar z,\bar r):\bar r\in[-BR,BR]^m,\ q_{U,\bar x}(\bar r)\le R\}.$$
	%This display describes the semantic value and is not an object-language scalar quantifier.
	
	Let $\omega$ be a common modulus for $\Phi$ in its numerical slots. The evaluation bound implies $\|\bar r\|_\infty\le Bq_{U,\bar x}(\bar r)$. If $q_{U,\bar x}(\bar r)>R$, radial contraction to $\bar s=R\bar r/q_{U,\bar x}(\bar r)$ gives $q_{U,\bar x}(\bar s)=R$ and $\|\bar r-\bar s\|_\infty\le B(q_{U,\bar x}(\bar r)-R)$. It follows that the constrained infimum equals the unconstrained minimum on the numerical cube of
	$$\Phi(\bar z,\bar r)+\omega\bigl(B(q_{U,\bar x}(\bar r)\dotminus R)\bigr).$$
	Choose finite rational grids $G_n\subseteq[-BR,BR]^m$ whose mesh tends to zero. Uniform continuity of $\Phi$ and Theorem~\ref{thm:uniform-definability-interpolation-gauge} show that the finite expressions
	$$\min_{\bar a\in G_n}\left[\Phi(\bar z,\bar a)+\omega\bigl(B(q_{U,\bar x}(\bar a)\dotminus R)\bigr)\right]$$
	converge uniformly to $\Psi$. Each $\bar a$ here is a finite tuple of numerical constants inside continuous connectives, so these expressions contain no scalar quantifier. Supremum quantifiers are handled by applying the same construction to a bounded affine complement of $\Phi$. Repeating the translation eliminates all bounded harmonic quantifiers in the fragment.
\end{proof}

\begin{corollary}[Relative model-completeness]
	\label{cor:relative-model-completeness}
	Let $\mathcal L_0$ be the reduct consisting of the point sorts and the non-harmonic symbols, and let $\Delta_{0}$ be the bounded
	finite evaluation fragment.  If $M\subseteq N\models\mathrm{BHS}+\mathrm{FI}$ and
	$M|_{\mathcal L_0}\preccurlyeq N|_{\mathcal L_0}$, then $M\preccurlyeq_{0}N$.
	Equivalently, every embedding of models whose $\mathcal L_0$-reduct is elementary is elementary on the bounded finite evaluation fragment.
\end{corollary}

\begin{proof}
	By the preceding elimination theorem, every formula in
	$\Delta_{0}$ is equivalent to one without bounded harmonic
	quantifiers.  Its value is therefore preserved by
	$\mathcal L_0$-elementarity.
\end{proof}

\begin{proposition}[Euclidean harmonic-polynomial interpolation]
	\label{prop:harmonic-polynomial-FI}
	Let $n\ge2$, let $U\subseteq\mathbb R^n$ be open, and let $x_1,\ldots,x_m\in U$ be pairwise distinct. For every numerical tuple $\bar r\in\mathbb R^m$, there is a real harmonic polynomial $p$ of degree at most $m-1$ such that $p(x_i)=r_i$ for every $i$. Consequently, the classical Euclidean harmonic structure satisfies FI and the uniform conclusion of Lemma~\ref{lem:uniform-finite-interpolation}. Its interpolants may be chosen among restrictions of total harmonic polynomials of degree at most $m-1$.
\end{proposition}

\begin{proof}
	Choose orthonormal vectors $a,b\in\mathbb R^n$ so that the complex numbers $z_i=\langle a,x_i\rangle+\mathrm{i}\langle b,x_i\rangle$ are pairwise distinct. Such a pair exists because only finitely many proper closed subsets of the Grassmannian, determined by the differences $x_i-x_j$, must be avoided. Let $P\in\mathbb C[z]$ be the Lagrange interpolation polynomial with $P(z_i)=r_i$. Then $\deg P\le m-1$, and
	$p(x)=\operatorname{Re}P(\langle a,x\rangle+\mathrm{i}\langle b,x\rangle)$
	is a real harmonic polynomial with the prescribed values. Applying Lemma~\ref{lem:uniform-finite-interpolation} inside this finite-degree family gives the uniform gauge bound. In dimension one, harmonic functions are affine, so interpolation at three arbitrary points fails.
\end{proof}

Let $\mathcal R$ be an o-minimal expansion of the real field, and let $H_{\mathcal R}(\mathbb R^n)$ denote the total $\mathcal R$-definable harmonic functions. Miller \cite{Miller2026blms} states that every member of this family is a polynomial. Thus $H_{\mathcal R}(\mathbb R^n)=\bigcup_{d<\omega}\mathscr H_{\le d}(\mathbb R^n)$, where $\mathscr H_{\le d}(\mathbb R^n)$ is the space of harmonic polynomials of degree at most $d$.

\begin{corollary}[Uniform FI for o-minimal total harmonic functions]
	\label{cor:FI-o-minimal-total-harmonic}
	Let $n\ge2$. The family $H_{\mathcal R}(\mathbb R^n)$ satisfies FI and Uniform FI on every separated compact configuration set. Every $m$-point finite evaluation image is already realized by $\mathscr H_{\le m-1}(\mathbb R^n)$, so Theorem~\ref{thm:relative-finite-profile-QE} applies with harmonic polynomial witnesses.
\end{corollary}

\begin{proof}
	Proposition~\ref{prop:harmonic-polynomial-FI} supplies degree at most-$m-1$ interpolants, and Lemma~\ref{lem:uniform-finite-interpolation} makes their gauges uniform on separated compact configuration sets. Miller's theorem identifies all total $\mathcal R$-definable harmonic functions with harmonic polynomials, so the bounded finite evaluation images are governed by the same quotient gauges.
\end{proof}

Borgard and Miller \cite{Borgard_Miller_2026} prove that, for each exponential term $f$, there is a degree $d(f)$ bounding every harmonic partial specialization of $f$. Hence these specializations have finite evaluation images contained in the image of the finite-dimensional space $\mathscr H_{\le d(f)}(\mathbb R^n)$. This uniform degree control complements the uniform gauge control established above.

\section{Superharmonic functions and balayage}

\subsection{Superharmonic functions}
Let $\mathfrak M\models\mathrm{BHS}$ be an $\mathcal L_{\mathrm{Har}}$-structure.

\begin{definition}[Superharmonic function] Let $U\subseteq X$ be a domain. A half-extended real-valued function $u:U\to(-\infty,+\infty]$ is superharmonic if
	\begin{enumerate}
		\item $u$ is lower semicontinuous,
		\item $u$ is not identically $+\infty$ on any connected component of $U$,
		\item for every relatively compact regular open set $V\Subset U$ and every $h\in H(V)\cap C(\overline V)$, if $h\le u$ on $\partial V$,
		then
		$h\le u$ on $V$.
	\end{enumerate}
	
	The superharmonic functions on $U$ form $S(U)$, and $S^+(U)$ denotes its nonnegative cone.
\end{definition}

\begin{definition}[Least superharmonic majorant]
	The least superharmonic majorant of a bounded function $f$ is the least member of $S(U)$ that dominates $f$, whenever such a least member exists.
\end{definition}

\begin{definition}[Reduction and balayage]
	Let $A\subseteq U$ and let $u$ be a nonnegative extended-valued function on $U$. Its reduction onto $A$ is
	$$
	R_u^A(x)
	=
	\inf\{v(x):v\in S(U),\ v\ge u\text{ on }A\}.
	$$
	Its lower-semicontinuous regularization $\widehat R_u^A(x)=\liminf_{y\to x}R_u^A(y)$ is the balayage of $u$ onto $A$ and is denoted by $\operatorname{Bal}_A(u)$.
\end{definition}

\subsection{Expanded language Bal} 
Call a subset admissible when its code and the restricted infima used below are named in the chosen structure.

\begin{definition}[Language Bal] The language Bal is
	$$\mathcal L_{\mathrm{Bal}}=\mathcal L_{\mathrm{Har}}\cup\{S(U):U\in\tau_0\}\cup\{\operatorname{Bal}_A:A\subseteq X\text{ is admissible}\}.$$
\end{definition}

In classical potential theory, %the reduction operator $R^A_u$ sends a bounded function $u$ to the infimum of all superharmonic functions majorizing $u$ on $A$, and 
the Dirichlet solution operator can be reconstructed from the lower-semicontinuous regularization $\operatorname{Bal}_{\partial U}(\tilde g)|_U$ of a bounded extension of $g$. The following lemma makes this observation precise.

\begin{lemma}[Dirichlet solution operator via balayage]\label{prop: HU via balayage}
	Let $U\in\tau_0$ and $g\in C(\partial U)$. Choose a relatively compact open neighborhood $W$ of $\overline U$ and a bounded continuous extension $\tilde g$ to $\overline W$. Then
	$$
	D_U(g)=\operatorname{Bal}_{\partial U}(\tilde g)|_U.
	$$
	Here the reduction and balayage are taken in $W$.
\end{lemma}

\begin{proof}
	Since $\overline W$ is compact Hausdorff, the Tietze extension theorem extends every $g\in C(\partial U)$ to a continuous function on $\overline W$. Put $v=\operatorname{Bal}_{\partial U}(\tilde g)$. The classical off-support theorem for balayage shows that $v$ is harmonic on $W\setminus\partial U$, hence on $U$.
	
	Fix $\xi\in\partial U$ and $\varepsilon>0$. By continuity, $|g(\eta)-g(\xi)|<\varepsilon$ on a boundary neighborhood of $\xi$. Regularity of $U$ provides a positive barrier $b_\xi$ with $b_\xi(x)\to0$ as $x\to\xi$ from $U$ and with a positive lower bound on the remainder of $\partial U$. If $M$ dominates $|g-g(\xi)|$ on $\partial U$, then for a sufficiently large constant $c$ the two superharmonic functions $g(\xi)+\varepsilon+cb_\xi$ and $-g(\xi)+\varepsilon+cb_\xi$ majorize respectively $g$ and $-g$ on $\partial U$. The defining comparison property of reduction therefore yields $|v(x)-g(\xi)|\le\varepsilon+cb_\xi(x)$ near $\xi$. Letting $x\to\xi$ and then $\varepsilon\downarrow0$ gives $v(x)\to g(\xi)$.
	
	Thus $v|_U$ is harmonic and extends continuously to every boundary point with trace $g$. Uniqueness of the regular Dirichlet problem gives $v|_U=D_U(g)$.
\end{proof}

%The fine topology is the finest topology such that every superharmonic function is continuous. 
Let $U\in\tau_0$. A function $p\in S^+(U)$ is called a {potential} if its greatest nonnegative harmonic minorant is zero. %This is $\lim_{x\to\xi}p(x)=0$, $ \xi\in\partial U$ whenever the relevant boundary limit exists, 
%Indeed, boundary vanishing forces every nonnegative harmonic minorant to vanish by the maximum principle, while the converse follows from the boundary characterization of potentials on regular domains.

\begin{proposition}[Riesz decomposition]
	\label{cor:riesz-decomposition}
	Let $U$ be Greenian and $u\in S^+(U)$. Then there are unique
	$h\in H^+(U)$ and a potential $p\in S^+(U)$ such that
	$u=h+p$. Moreover, there is a unique positive Radon measure
	$\lambda_p$ on $U$ satisfying
	$p = G_U *\lambda_p.$
	If $p$ is harmonic on $U\setminus E$ for a closed set $E$, then
	$\operatorname{supp}(\lambda_p)\subseteq E$.
\end{proposition}

\begin{proof}
	Fix $x_0\in U$ with $u(x_0)<\infty$. For each relatively compact open set
	$V\Subset U$ containing $x_0$, define $h_V = \widehat R_u^{U\setminus V}$.
	Then $0\le h_V\le u$, $h_V$ is harmonic on $V$, and the net $(h_V)$ decreases
	as $V$ increases. By Harnack compactness for directed families \cite[Lemma~2]{CC1963axiomatic},
	$h_V$ converges to $h$ locally uniformly, where $h=\inf_V h_V$ is positive harmonic on $U$.
	
	Take $v\in H^+(U)$ with $v\le u$. For any $V$, every superharmonic
	function $s$ appearing in the definition of $h_V$ satisfies
	$s\ge u\ge v$ on $U\setminus V$, and hence $s\ge v$ on all of $U$
	by the minimum principle. Taking the infimum over such $s$ gives
	$h_V\ge v$, and the limit yields $h\ge v$. Thus $h$ is the largest
	harmonic minorant.
	
	Set $p=u-h$. If $p$ had a nonzero harmonic minorant $w$, then
	$h+w$ would be a harmonic minorant of $u$ strictly larger than $h$,
	a contradiction. Hence $p$ is a potential. The Riesz representation theorem for potentials supplies a unique positive Radon
	measure $\lambda_p$ with $p = G_U * \lambda_p$, and
	$\operatorname{supp}(\lambda_p)\subseteq E$ when $p$ is harmonic
	on $U\setminus E$.
	
	Let $u = h' + p'$ be another decomposition.
	Then $h'\le h$, so $h-h'$ is a harmonic minorant of the potential
	$p'$. This forces $h=h'$, and consequently $p=p'$, proving uniqueness.
\end{proof}

\subsection{Types and local types}
Let $T$ be an $\mathcal L$-theory, let $\mathfrak M\models T$ be an $\mathcal L$-structure, and let $A\subseteq M$. It is convenient to expand the language
$$\mathcal L_A = \mathcal L \cup \{c_a: a \in A\}$$
by adding constant symbols $c_a,a\in A$. To display parameters explicitly, we write an $\mathcal L_A$-formula as $\varphi(\bar y;\bar a)$ with free variable $\bar y$ and parameter $\bar a$ in some copy of $A$.  

\begin{definition}[Partial and complete types] A {partial $n$-type} $p(\bar x)$ over $A$ is a family of $\mathcal L_A$-formulas with common free-variable tuple $\bar x$ of length $n$ that is consistent with the full $\mathcal L_A$-theory $\operatorname{Th}_A(\mathfrak M)$. A {complete $n$-type} over $A$ is a maximal consistent partial type over $A$. The space of all complete $n$-types over $A$ is denoted by $S_n(A)$.
\end{definition}

\begin{definition}[Realization and omission]
	Let $\mathfrak M$ be an $\mathcal L$-structure. For an $\mathcal L$-formula $\varphi(\bar x)$ with $|\bar x|=n$, a tuple
	$\bar a\in M^{n}$ is said to realize $\varphi$ in
	$\mathfrak M$ if
	$\varphi^{\mathfrak M}(\bar a)=0$.
	The formula $\varphi$ is realized in $\mathfrak M$ if such a tuple exists, and is omitted if no such tuple exists.
	%It is uniformly omitted if $$    \inf_{\bar a\in M^{|\bar x|}}    \varphi^{\mathfrak M}(\bar a)>0.$$
	The structure $\mathfrak M$ realizes a partial type $p(\bar x)$ if there exists a tuple $\bar a\in M^n$ that satisfies every condition in $p(\bar x)$. Otherwise, $\mathfrak M$ omits $p(\bar x)$.
\end{definition}

\begin{definition}[Local type]
	Given a set $\Delta$ of $\mathcal L_A$-formulas in the common tuple of free variables $\bar x$ of length $n$, a {local type} over a parameter set $A$ is a complete type determined by conditions of the form
	$$\varphi(\bar x)<r,\quad\varphi(\bar x)=r,\quad\varphi(\bar x)>r,$$
	where $\varphi\in\Delta$ and $r\in\mathbb Q_{>0}$. The space of all such local types is denoted by $S_{\Delta}(A)$. When $\Delta=\{\varphi\}$ consists of a single $\mathcal L$-formula, we write $S_{\{\varphi\}}(A)=S_\varphi(A)$. For $p\in S_n(A)$, let $p\upharpoonright\varphi$ denote the collection of $\varphi$-conditions in $p$ of these three forms.
\end{definition}

\begin{definition}[Principal type]
	A (local) $n$-type $p(\bar x)$ over $A$ is called {principal} if there exist an $\mathcal L_A$-formula $\psi(\bar y)$, terms $t_1(\bar y),\dots,t_n(\bar y)$, and a rational number $r\in\mathbb Q_{>0}$ such that
	\begin{itemize}
		\item $\operatorname{Th}_A(\mathfrak M)\cup\{\psi(\bar y)=0\}$ is satisfiable,
		\item $\operatorname{Th}_A(\mathfrak M)\cup\{\psi(\bar y)\le r\}\models p(t_1(\bar y),\dots,t_n(\bar y))$, i.e. in every model of $\operatorname{Th}_A(\mathfrak M)$, whenever $\psi(\bar b)\le r$ holds, the tuple $(t_1(\bar b),\dots,t_n(\bar b))$ realizes $p(\bar{x})$.
	\end{itemize}
	Write $\bar t(\bar y)=(t_1(\bar y),\ldots,t_n(\bar y))$. We then say that the condition $\psi(\bar y)\le r$ and the $\mathcal L_A$-terms $\bar t$ {witness} the principality of $p(\bar x)$.
\end{definition}

\begin{lemma}[Realization of principal types]
	\label{lem:realization-principal-types}
	Every principal local type over $A$ with a positive witnessing
	radius is realized in every model of
	$\operatorname{Th}_{A}(\mathfrak M)$. If $|\mathcal L_A| = \aleph_0$, the converse is also true.
\end{lemma}

\begin{proof}
	Let $p(x)$ be principal, witnessed by an
	$\mathcal L_{A}$-formula $\psi(\bar y)$, an
	$\mathcal L_{A}$-term $t(\bar y)$, and a rational $r>0$. The
	satisfiability of
	$\operatorname{Th}_{A}(\mathfrak M)\cup\{\psi(\bar y)=0\}$
	means that $\inf_{\bar y}\psi(\bar y)$ has value zero in the complete
	theory. If
	$\mathfrak N\models\operatorname{Th}_{A}(\mathfrak M)$, then
	$\inf_{\bar y}\psi^{\mathfrak N}(\bar y)=0$. Since $r>0$, there is
	$\bar b\in\mathfrak N$ such that $\psi^{\mathfrak N}(\bar b)\le r$.
	The witnessing property gives $t^{\mathfrak N}(\bar b)\models p$.
	
	Conversely, suppose that $p$ is realized in every model of
	$\operatorname{Th}_{A}(\mathfrak M)$. If $p$ is not principal with
	a positive witnessing radius, then the continuous omitting type
	theorem \cite[Theorem 4.9]{Eagle2014omiting} would yield a model
	$\mathfrak N\models\operatorname{Th}_{A}(\mathfrak M)$ omitting
	$p$. This contradicts the hypothesis that $p$ is realized in every
	such model. Therefore $p$ is principal and admits a positive
	witnessing radius.
\end{proof}

For $p\in S_\varphi(A)$, let
$\varphi^p(a)$ denote the unique value assigned to the instance $\varphi(y;a)$ by the completeness of $p$. The local logic topology $\tau_\varphi$ on $S_\varphi(A)$ is the
weakest topology for which, for every $a\in A$, the evaluation map
$e_a:p\mapsto\varphi^p(a)$
is continuous. A basis of $\tau_\varphi$-open sets is given by finite
intersections of sets of the form
$$
[p;a,\varepsilon]_\varphi = \left\{
q\in S_\varphi(A):
\left|\varphi^p(a) - \varphi^q(a)\right|<\varepsilon
\right\}, \quad\varepsilon>0.
$$
There is a natural continuous restriction map $r_\varphi:p\mapsto p\upharpoonright\varphi$ from $S_1(A)$ to $S_\varphi(A)$.
Also, $\varphi$ induces a definable pseudo-metric on $S_\varphi(A)$,
%$$   d_\varphi(p,q) =    \inf\left\{r\ge0:\operatorname{Th}_A(\mathfrak M)\cup p(y)\cup q(z)\cup \{d_\varphi(y,z)\le r\}\text{ is consistent}\right\}.$$
$$d_{\varphi}(p,q)=\sup_{a\in A}\left|\varphi^p(a)-\varphi^q(a)\right|.$$
The topology induced by $d_{\varphi}$ is denoted by
$\tau_d$, which is generally finer than $\tau_\varphi$.
%Since uniform convergence on $A$ implies pointwise convergence,
\begin{definition}[Isolated type]
	A (local) type $p \in S_1(A)$ (or $S_\varphi(A)$) is called logic-isolated if $\{p\}$ is open in the logic topology $\tau$ (or $\tau_\varphi$).
\end{definition}

\section{Martin boundary and Keisler measures}

\subsection{The expanded Martin language}

Let $\tau_G\subseteq\tau_0$ be the family of relatively compact regular domains that are
Greenian. We assume that $\tau_G\neq\varnothing$ and that it generates the same topology as $\tau_0$. Define the {Greenian weight} of $X$ by
$$
{w}_G(X)
=
\min\left\{
|\tau_G'|:
\tau_G'\subseteq\tau_G
\text{ is a basis for the topology of }X
\right\}.
$$

For $U\in\tau_G$, fix a base point $x_0\in U$ and
write $U^\times=U\setminus\{x_0\}$. If $G_U$ is the Green kernel of
$U$, the normalized Martin kernel is
$K_{U}(x,y)=G_U(x,y)/G_U(x_0,y)$ for $y\in U^\times$. Thus
$K_{U}(x_0,y)=1$, and for a fixed pole $y$ the function
$K_{U}(\cdot,y)$ is positive and harmonic away from $y$. Choose a continuous strictly increasing map
$\beta:[0,\infty]\to[0,1]$ with $\beta(\infty)=1$. The bounded local
Martin predicate is
$\varphi_{U}(y;x)=\beta(K_{U}(x,y))$. We assume that this predicate has a uniformly continuous extension across the diagonal of $U$. 

\begin{definition}[Language Mar]
	The language Mar is
	$$
	\mathcal L_{\mathrm{Mar}}
	=
	\mathcal L_{\mathrm{Bal}}
	\cup
	\{\varphi_U:U\in\tau_G\}.
	$$
	Consequently,
	$|\mathcal L_{\mathrm{Mar}}|=|\mathcal L_{\mathrm{Bal}}|+w_G(X)+\aleph_0$.
	In particular, if $\mathcal L_{\mathrm{Bal}}$ is countable and
	$w_G(X)=\aleph_0$, then
	$\mathcal L_{\mathrm{Mar}}$ is countable.
\end{definition}

For $y\in U^\times$, let $\operatorname{Prof}_U(y)$ be the function $x\mapsto\varphi_U(y;x)$. Equip the realized profile set with the compact-open uniformity generated by the pseudometrics $d_C(y,z)=\sup_{x\in C}|\varphi_U(y;x)-\varphi_U(z;x)|$, where $C\Subset U$ is compact. Profiles are identified when all these pseudometrics vanish. If $U$ has a countable compact exhaustion $(C_n)_{n<\omega}$, this uniformity is metrized by $d_M(y,z)=\sum_{n<\omega}2^{-n-1}d_{C_n}(y,z)$; countability is not assumed in the general construction.

\begin{lemma}[Compactness of the Martin profile closure]
	\label{lem:martin-profile-compactness}
	For $U\in\tau_G$, let
	$$
	\mathcal P_U
	=
	\{\operatorname{Prof}_U(y):y\in U^\times\}
	\subseteq C(U,[0,1]).
	$$
	Then $\mathcal P_U$ has compact closure in the compact-open topology.
\end{lemma}

\begin{proof}
	Fix a compact set $C\Subset U$ and choose a relatively compact open
	set $W$ such that $C\cup\{x_0\}\subseteq W\Subset U$. For poles $y\in\overline W$, the assumed continuous extension of
	$\varphi_U$ across the diagonal is uniformly continuous on the compact
	set $C\times\overline W$. Hence $\{\operatorname{Prof}_U(y)|_C:y\in\overline W\}$ is equicontinuous.
	
	For $y\notin W$, the function $K_U(\,\cdot\,,y)$ is positive harmonic
	on $W$ and satisfies $K_U(x_0,y)=1$. Harnack compactness therefore
	implies that
	$\{\operatorname{Prof}_U(y)|_C:y\notin W\}$
	is equicontinuous as well. Thus the whole family
	$\mathcal P_U|_C$ is equicontinuous and takes values in the compact
	interval $[0,1]$.
	
	By the Arzel\`a--Ascoli theorem, its closure in $C(C,[0,1])$ is compact.
	Since this holds for every compact $C\Subset U$, the Arzel\`a--Ascoli
	theorem for the compact-open topology gives that
	$\overline{\mathcal P_U}$ is compact in $C(U,[0,1])$.
\end{proof}

\begin{definition}[Local Martin compactification]
	The local Martin compactification $\widehat U^{M}$ is the closure 
	of $\{\operatorname{Prof}_{U}(y):y\in U^\times\}$ under the compact-open topology of $C(U,[0,1])$. The Martin boundary is $\partial_MU=\widehat U^M\setminus U$.
\end{definition}

We assume that the interior map extends over $x_0$, embeds $U$ densely, and separates boundary profiles.

\begin{lemma}[Harmonic extension of boundary profiles]
	\label{prop:Martin-boundary-harmonic-profile}
	For every $\xi\in\partial_MU$, there is a unique normalized positive
	harmonic function $k_\xi\in H^+(U)$ such that $k_\xi(x_0)=1$ and
	$K_{U}(\cdot,y_i)\to k_\xi$ locally uniformly on $U$ whenever
	$y_i\to\xi$. Equivalently,
	$\operatorname{Prof}_{U}(\xi)(x)=\beta(k_\xi(x))$ for every
	$x\in U$.
\end{lemma}

\begin{proof}
	Let $(y_i)$ converge to $\xi\in\partial_MU$. For each compact
	$C\Subset U$, the poles eventually lie outside $C$, so
	$K_{U}(\cdot,y_i)$ is positive and harmonic on a neighborhood of
	$C$. Its value at $x_0$ is always $1$. The Harnack inequalities and
	the Harnack compactness theorem therefore give a locally uniformly
	convergent subnet whose limit is a positive harmonic function
	$k_\xi$ normalized at $x_0$ \cite[Theorem~1]{LW1965axiomatic}. The bounded profiles converge to
	the profile prescribed by $\xi$, so all convergent subnets have the
	same limit. This proves both convergence and uniqueness.
\end{proof}

\subsection{Martin boundary and non-principal types}
We work in the local expansion $\mathcal L_{\mathrm{Mar}}(U,x_0)=\mathcal L_{\mathrm{Har}}\cup\{x_0,\varphi_U\}$, so cardinal arithmetic gives $|\mathcal L_{\mathrm{Mar}}(U,x_0)|=\max\{|\mathcal L_{\mathrm{Har}}|,\aleph_0\}$. For a topological space $Y$, write $\operatorname{dens}(Y)=\min\{|D|:D\subseteq Y\text{ is dense in }Y\}$. Fix a dense set $A\subseteq U$ with $|A|=\operatorname{dens}(U)$ and write $Ax_0=A\cup\{x_0\}$. Continuity implies that every compact-open profile is determined by its values on $A$.

For $y\in U^\times$, put $p_y=\operatorname{tp}_\varphi(y/Ax_0)$ and $B_\varphi(Ax_0)=\{p_y:y\in U^\times\}$. Let $C_\varphi(Ax_0)$ be its closure in the local logic topology inherited from $[0,1]^{Ax_0}$. Its elements are the Martin types over $Ax_0$. We use the $Ax_0$-definitional expansion and assume that the bounded Martin profiles separate points of $U$ and have the compact-uniform continuity required above. The local language with parameters has cardinality
$$|\mathcal L_{\mathrm{Mar}}(U,x_0;A)|=\max\{|\mathcal L_{\mathrm{Har}}|,\operatorname{dens}(U),\aleph_0\}.$$

\begin{lemma}[Cauchy profiles]
	\label{lem:Martin-Cauchy-profile-type}
	Every compact-open Cauchy net of realized profiles determines a consistent complete local type over $Ax_0$. The resulting map from the profile closure to $C_\varphi(Ax_0)$ is injective and is a homeomorphism onto its image. If a countable compact exhaustion is fixed, this map preserves the profile metric.
\end{lemma}

\begin{proof}
	Let $(y_i)$ be a compact-open Cauchy net. For every $a\in Ax_0$, the values $\varphi_U(y_i;a)$ converge; denote the limit by $r_a$. Any finite collection of conditions $|\varphi(y;a)-r_a|\le\varepsilon$ is approximately realized by some tail element $y_i$. Continuous compactness therefore gives a complete local type with coordinates $(r_a)_{a\in Ax_0}$. Density of $A$ and continuity of profiles show that two compact-open limits with the same coordinates agree on every compact subset of $U$, proving injectivity. Coordinate convergence is exactly the local logic topology, while uniform convergence on compact sets is the profile topology. On the compact closure, the continuous bijection between these Hausdorff spaces is a homeomorphism. In the metrizable case, density gives equality of the suprema over $A\cap C_n$ and $C_n$, hence equality of the profile metrics.
\end{proof}

\begin{lemma}[The reference type]
	\label{lem:reference-Martin-type}
	The reference type $p_{x_0}$ belongs to $C_\varphi(Ax_0)$ and is
	principal but non-isolated.
\end{lemma}

\begin{proof}
	Because $x_0$ is named in the parameter set $Ax_0$, the constant term
	$c_{x_0}$ is available. Taking the identically zero formula, the
	constant term $c_{x_0}$, and any positive rational radius witnesses
	that $p_{x_0}$ is principal.
	
	Choose a net of distinct points $x_i\in U^\times$ converging to $x_0$. Compact-uniform continuity makes their profiles converge to the profile of $x_0$, so Lemma~\ref{lem:Martin-Cauchy-profile-type} gives $p_{x_0}\in C_\varphi(Ax_0)$.
	
	The same convergence is coordinatewise on every parameter in $Ax_0$,
	so it implies convergence in the local logic topology. Since the
	profiles separate points, $p_{x_i}\ne p_{x_0}$. Every logic
	neighborhood of $p_{x_0}$ therefore contains another Martin type, and
	$p_{x_0}$ is non-isolated.
\end{proof}

\begin{lemma}[Principal Martin types]
	\label{lem:principal-Martin-types}
	The principal Martin types in $C_\varphi(Ax_0)$ are precisely
	$\{p_x:x\in U\}=B_\varphi(Ax_0)\cup\{p_{x_0}\}$.
\end{lemma}

\begin{proof}
	Fix $x\in U$. Since $A$ is dense, choose $a_m\in A$ with $a_m\to x$.
	The inequality
	$|d_U(y,a_m)-d_U(y,x)|\le d_U(a_m,x)$ shows uniform convergence to
	$d_U(y,x)$. Hence $x\in\operatorname{dcl}(Ax_0)$, and the
	$Ax_0$-definitional expansion contains a constant term $c_x$. The
	identically zero formula, $c_x$, and any positive rational radius
	witness that $p_x$ is principal.
	
	Conversely, let $p\in C_\varphi(Ax_0)$ be principal. By
	Lemma~\ref{lem:realization-principal-types}, $p$ is realized in the
	point sort $U$ of $\mathfrak M$. Thus
	$p=\operatorname{tp}_\varphi(x/Ax_0)=p_x$ for some $x\in U$. Finally,
	$U=U^\times\sqcup\{x_0\}$ and the displayed equality follows from the
	definition of $B_\varphi(Ax_0)$.
\end{proof}

\begin{theorem}[First isomorphism theorem]
	\label{thm:global-Martin-profile-correspondence}
	The map $x\mapsto p_x$ on $U^\times$ extends uniquely to a
	homeomorphism
	$\Gamma:\widehat U^{M}\to C_\varphi(Ax_0)$. Under this
	correspondence, $p_{x_0}$ is the unique principal type not realized in $U^\times$, while $\partial_MU$
	corresponds exactly to the non-principal Martin types over $Ax_0$.
	When $U$ has a countable compact exhaustion and the metric above is used, $\Gamma$ is an isometry.
\end{theorem}

\begin{proof}
	Lemma~\ref{lem:Martin-Cauchy-profile-type} extends the interior map to a topological embedding of the compact profile closure into $C_\varphi(Ax_0)$. Its image is closed and contains the dense set $B_\varphi(Ax_0)$, so it is all of $C_\varphi(Ax_0)$. Lemmas~\ref{lem:reference-Martin-type} and \ref{lem:principal-Martin-types} identify the image of $U$ with the principal locus. Its complement is therefore the non-principal locus. The last statement is the metric clause of Lemma~\ref{lem:Martin-Cauchy-profile-type}.
\end{proof}

For $\xi\in\widehat U^{M}$, we write $p_\xi$ for the local type over $Ax_0$ determined by $\varphi^{p_\xi}(a)=\operatorname{Prof}_{U}(\xi)(a)$ for $a\in A$, together with its value at the named base point.

\begin{corollary}[Realization in an elementary saturation]
	\label{thm:saturated-Martin-realization}
	Let $\kappa$ be a regular cardinal such that
	$\kappa>\max\{|\mathcal L_{\mathrm{Har}}|,\operatorname{dens}(U),\aleph_0\}$.
	Then every boundary type $p_\xi\in S_\varphi(Ax_0)$, with
	$\xi\in\partial_MU$, is realized in any sufficiently large
	$\kappa$-saturated and strongly $\kappa$-homogeneous elementary
	extension $\mathfrak C\succeq\mathfrak M$.
\end{corollary}

\begin{proof}
	The preceding cardinal estimate gives $|Ax_0|<\kappa$ and
	$|\mathcal L_{\mathrm{Mar}}(U,x_0)|<\kappa$. Hence
	$\kappa$-saturation realizes every complete local type over $Ax_0$.
\end{proof}

%\begin{definition}[Minimal Martin boundary] 
%A nonzero positive harmonic function $h\in H^+(X)$ is called {minimal} if for every $ v\in H^+(X)$, $  0\le v\le h$ implies $    v=ch$ for some $c\in[0,1]$. Equivalently, $  \mathbb R_{\ge0}h$ is an extreme ray of the cone $H^+(X)$. 
%The minimal Martin boundary is$$    \partial_mX    =    \left\{        \xi\in\partial_{M}X:        k_\xi        \text{ is a minimal positive harmonic function}    \right\}.$$\end{definition}

\subsection{Keisler measures and minimal types}

Let
$$
\partial_M^{\mathrm{np}}(Ax_0)\subseteq C_\varphi(Ax_0)
$$
denote the non-principal locus. Theorem~\ref{thm:global-Martin-profile-correspondence}
restricts to a homeomorphism
$$
\Gamma:\partial_MU\to\partial_M^{\mathrm{np}}(Ax_0).
$$
Let $\operatorname{Ev}(h;x)=h(x)$.

\begin{definition}[Positive harmonic evaluation type cone]
\label{def:positive-harmonic-evaluation-cone}
Define
$$\mathscr H^+(Ax_0)=\{\operatorname{tp}_{\operatorname{Ev}}(h/Ax_0):h\in H^+(U)\}$$
and let $0=\operatorname{tp}_{\operatorname{Ev}}(0/Ax_0)$. Its normalized section is
$$\mathscr H_1(Ax_0)=\{q\in\mathscr H^+(Ax_0):\operatorname{Ev}^q(x_0)=1\}.$$
The order $r\le q$, addition, and multiplication by $c\ge0$ are external pointwise operations on evaluation functions: for example, $\operatorname{Ev}^{cq}(a)=c\operatorname{Ev}^q(a)$ for $a\in Ax_0$. These operations introduce no scalar sort or scalar quantifier.
\end{definition}

Density of $A$ and continuity of harmonic functions show that an evaluation type in $\mathscr H^+(Ax_0)$ determines its harmonic representative uniquely. Consequently, pointwise inequalities on $Ax_0$ extend to all of $U$.

\begin{lemma}[Compactness and topology of normalized harmonic types]
\label{lem:compact-convex-harmonic-profiles}
The set $\mathscr H_1(Ax_0)$ is compact and convex. Its evaluation-type topology, generated by the coordinates $q\mapsto\operatorname{Ev}^q(a)$ for $a\in Ax_0$, agrees with the topology transported from compact-open convergence on $H_1^+(U,x_0)=\{h\in H^+(U):h(x_0)=1\}$.
\end{lemma}

\begin{proof}
Harnack's inequalities bound $H_1^+(U,x_0)$ uniformly on every compact subset of $U$. The Harnack compactness theorem then gives a compact-open convergent subnet of every net, and its limit remains positive harmonic and has value one at $x_0$ \cite[Theorem~1]{LW1965axiomatic}. Thus $H_1^+(U,x_0)$ is compact in the compact-open topology.

Restriction to $Ax_0$ is continuous from $H_1^+(U,x_0)$ to the Hausdorff product space $[0,\infty)^{Ax_0}$, and it is injective because $A$ is dense. A continuous injection from a compact space into a Hausdorff space is a homeomorphism onto its image. That image is exactly $\mathscr H_1(Ax_0)$, which proves compactness and identifies the two topologies. Convexity follows from the external pointwise convex operations on normalized positive harmonic functions.
\end{proof}

\begin{definition}[Minimal positive harmonic ray]
\label{def:minimal-positive-harmonic-ray}
For $q\in\mathscr H^+(Ax_0)\setminus\{0\}$, the ray $\mathbb R_{\ge0}q$ is minimal if every $r\in\mathscr H^+(Ax_0)$ satisfying $0\le r\le q$ has the form $r=cq$ for some $c\in[0,1]$. We then also call $q$ minimal.
\end{definition}

\begin{lemma}[Minimal rays and extreme points]
\label{lem:minimal-type-extreme}
For $q\in\mathscr H_1(Ax_0)$, the ray $\mathbb R_{\ge0}q$ is minimal if and only if $q\in\operatorname{Ext}(\mathscr H_1(Ax_0))$.
\end{lemma}

\begin{proof}
Suppose that $q$ is minimal and $q=tq_1+(1-t)q_2$, where $0<t<1$ and $q_1,q_2\in\mathscr H_1(Ax_0)$. Since $0\le tq_1\le q$, minimality gives $tq_1=cq$. Evaluation at $x_0$ yields $c=t$, hence $q_1=q$; the same argument gives $q_2=q$. Thus $q$ is extreme.

Conversely, suppose that $q$ is extreme and $0\le r\le q$. Put $c=\operatorname{Ev}^r(x_0)$. If $c=0$, the positive harmonic representative of $r$ vanishes at the interior point $x_0$, so the minimum principle, equivalently the Harnack principle, gives $r=0$ \cite[Theorem~1]{LW1965axiomatic}. If $c=1$, apply the same argument to the positive harmonic type $q-r$, whose value at $x_0$ is zero, and obtain $r=q$. When $0<c<1$, both $q_1=c^{-1}r$ and $q_2=(1-c)^{-1}(q-r)$ belong to $\mathscr H_1(Ax_0)$ and $q=cq_1+(1-c)q_2$. Extremality gives $q_1=q$, hence $r=cq$.
\end{proof}

For $p\in\partial_M^{\mathrm{np}}(Ax_0)$, write $k_p=k_{\Gamma^{-1}(p)}$ for the normalized positive harmonic function given by Lemma~\ref{prop:Martin-boundary-harmonic-profile}, and define
$$\mathfrak h:\partial_M^{\mathrm{np}}(Ax_0)\longrightarrow\mathscr H_1(Ax_0),\quad
\mathfrak h(p)=\operatorname{tp}_{\operatorname{Ev}}(k_p/Ax_0).$$

\begin{proposition}[Boundary kernels generate positive harmonic functions]
\label{lem:Martin-kernel-generation}
Let $U$ be Greenian and fix $x_0\in U$. In the compact-open topology, the closed positive harmonic cone satisfies
$H^+(U)=\overline{\operatorname{cone}}\{k_\xi:\xi\in\partial_MU\}$.
Equivalently, its normalized section satisfies
$H_1^+(U,x_0)=\overline{\operatorname{co}}\{k_\xi:\xi\in\partial_MU\}$.
\end{proposition}

\begin{proof}
Fix $h\in H_1^+(U,x_0)$ and direct the relatively compact open sets $V$ with $x_0\in V\Subset U$ by inclusion. Put $p_V=\widehat R_h^{\overline V}$. Since $\overline V$ is relatively compact, $p_V$ is a potential \cite[Section~21, Proposition~10]{Brelot1960}.
By the defining property of reduction, $p_V=h$ on $V$, while the
off-support theorem for balayage implies that $p_V$ is harmonic on
$U\setminus\overline V$. Hence $p_V$ is harmonic on
$U\setminus\partial V$. Proposition~\ref{cor:riesz-decomposition} therefore gives a unique
positive Radon measure $\lambda_V$ such that
$p_V=G_U*\lambda_V,
\operatorname{supp}(\lambda_V)\subseteq\partial V.$

Define $\mathrm{d}\mu_V(y)=G_U(x_0,y)\mathrm{d}\lambda_V(y)$.
Since $p_V(x_0)=h(x_0)=1$, each $\mu_V$ is a probability measure.
After passing to a subnet, $\mu_V$ converges weak-$*$ to a probability
measure $\mu$. Since $\partial V$ eventually leaves every compact subset
of $U$, $\mu$ is concentrated on $\partial_MU$. By
Lemma~\ref{prop:Martin-boundary-harmonic-profile}, the Martin kernels
extend to the boundary as the functions $k_\xi$ in the compact-open
topology. Passing to the limit therefore gives
$h(x)=\int_{\partial_MU}k_\xi(x)\,\mathrm{d}\mu(\xi)$ for every $x\in U$.
Hence $h$ is a barycenter of $\{k_\xi:\xi\in\partial_MU\}$ and so belongs
to its closed convex hull \cite[Proposition~1.1]{Phelps2001Choquet}. The reverse inclusion follows from the closed
convexity of $H_1^+(U,x_0)$.

Finally, every nonzero $g\in H^+(U)$ has $g(x_0)>0$ by the minimum
principle, and $g/g(x_0)\in H_1^+(U,x_0)$. Scaling the normalized identity
gives the conical identity.
\end{proof}

\begin{lemma}[Topology of the kernel-type map]
\label{prop:kernel-type-homeomorphism}
The map $\mathfrak h$ is a homeomorphism from $\partial_M^{\mathrm{np}}(Ax_0)$ onto its image in $\mathscr H_1(Ax_0)$. Its restriction identifies
$$\partial_m^{\mathrm{np}}(Ax_0):=\mathfrak h^{-1}\bigl(\operatorname{Ext}(\mathscr H_1(Ax_0))\bigr)$$
homeomorphically with $\operatorname{Ext}(\mathscr H_1(Ax_0))$. Moreover, $\Gamma(\partial_mU)=\partial_m^{\mathrm{np}}(Ax_0)$.
\end{lemma}

\begin{proof}
Through $\Gamma$, convergence in
$\partial_M^{\mathrm{np}}(Ax_0)$ is compact-open convergence of
$\beta(k_p)$. Harnack bounds on compact subsets of $U$, together
with uniform continuity of $\beta^{-1}$ on bounded intervals, show
that this is equivalent to compact-open convergence of $k_p$.
Lemma~\ref{lem:compact-convex-harmonic-profiles} therefore implies
that $\mathfrak h$ is a homeomorphism onto its image. Injectivity
follows from separation by Martin profiles.

Let
$B=\mathfrak h(\partial_M^{\mathrm{np}}(Ax_0))$.
By Proposition~\ref{lem:Martin-kernel-generation}, transported through the
affine homeomorphism
$\Theta:H_1^+(U,x_0)\to\mathscr H_1(Ax_0)$,
we have
$\mathscr H_1(Ax_0)=\overline{\operatorname{co}}\,B$.
The set $B$ is compact, so Milman's theorem gives
$\operatorname{Ext}(\mathscr H_1(Ax_0))\subseteq B$
\cite[Proposition~1.5]{Phelps2001Choquet}.

Now $\xi\in\partial_mU$ if and only if $k_\xi$ generates a minimal
positive harmonic ray, which by
Lemma~\ref{lem:minimal-type-extreme} is equivalent to
$\mathfrak h(\Gamma(\xi))\in
\operatorname{Ext}(\mathscr H_1(Ax_0))$.
Hence
$\mathfrak h(\Gamma(\partial_mU))
=\operatorname{Ext}(\mathscr H_1(Ax_0))$,
and injectivity of $\mathfrak h$ gives
$\Gamma(\partial_mU)=\partial_m^{\mathrm{np}}(Ax_0)$.
The first paragraph then yields the asserted homeomorphism on the
minimal loci.
\end{proof}

Let $\iota:\partial_m^{\mathrm{np}}(Ax_0)\hookrightarrow S_\varphi(Ax_0)$ be the inclusion. A local Keisler probability measure will mean a regular Borel probability measure on $S_\varphi(Ax_0)$.

\begin{theorem}[Second isomorphism theorem]
\label{thm:second-Martin-isomorphism}
Assume that the compact convex space $\mathscr H_1(Ax_0)$ is a metrizable Choquet simplex. Then every $q\in\mathscr H_1(Ax_0)$ has a unique local Keisler probability measure $\nu_q$ concentrated on $\partial_m^{\mathrm{np}}(Ax_0)$ such that, for every $a\in Ax_0$,
$$\operatorname{Ev}^q(a)=\int_{\partial_m^{\mathrm{np}}(Ax_0)}\operatorname{Ev}^{\mathfrak h(p)}(a)\,\mathrm{d}\nu_q(p).$$
The correspondence $q\mapsto\nu_q$ is affine. Its inverse sends a probability measure concentrated on the minimal Martin type locus to the barycenter in $\mathscr H_1(Ax_0)$ of its pushforward by $\mathfrak h$.

If $\partial_m^{\mathrm{np}}(Ax_0)$ is compact, then $\mathscr H_1(Ax_0)$ is a Bauer simplex and the correspondence is an affine homeomorphism for the evaluation-type topology and the weak-$*$ topology.
\end{theorem}

\begin{proof}
By Lemma~\ref{prop:kernel-type-homeomorphism}, the map $\mathfrak h$ identifies
$\partial_m^{\mathrm{np}}(Ax_0)$ homeomorphically with
$\operatorname{Ext}(\mathscr H_1(Ax_0))$.
The metrizable Choquet theorem gives, for each
$q\in\mathscr H_1(Ax_0)$, a probability measure $\mu_q$
concentrated on this extreme boundary with barycenter $q$, and the
simplex hypothesis makes $\mu_q$ unique
\cite[Section~3, Proposition~10.10]{Phelps2001Choquet}.
Set
$\nu_q=\iota_*(\mathfrak h^{-1})_*\mu_q$.
Then $\nu_q$ is a local Keisler probability measure concentrated on
$\partial_m^{\mathrm{np}}(Ax_0)$.

For $a\in Ax_0$, the map
$F_a(r)=\operatorname{Ev}^r(a)$ is continuous and affine, so the
barycenter identity gives
$\operatorname{Ev}^q(a)
=\int\operatorname{Ev}^{\mathfrak h(p)}(a)\,\mathrm{d}\nu_q(p)$.
Conversely, pushing a probability measure concentrated on
$\partial_m^{\mathrm{np}}(Ax_0)$ forward by $\mathfrak h$ and taking
its barycenter gives an element of $\mathscr H_1(Ax_0)$.
Choquet uniqueness shows that these constructions are inverse and
affine.

If $\partial_m^{\mathrm{np}}(Ax_0)$ is compact, then
$\operatorname{Ext}(\mathscr H_1(Ax_0))$ is compact by
Lemma~\ref{prop:kernel-type-homeomorphism}, so
$\mathscr H_1(Ax_0)$ is Bauer. The barycenter map is then a
continuous affine bijection from the compact space
$\mathcal P(\partial_m^{\mathrm{np}}(Ax_0))$ onto the Hausdorff space
$\mathscr H_1(Ax_0)$, and hence an affine homeomorphism.
\end{proof}

\begin{corollary}[Existence and compactness of complete lifts]
\label{cor:universal-lifting-Keisler-measure}
Under the hypotheses of Theorem~\ref{thm:second-Martin-isomorphism}, every $q\in\mathscr H_1(Ax_0)$ admits a probability measure $\mu$ on $S_1(Ax_0)$ such that $(r_\varphi)_*\mu=\nu_q$. Equivalently, $\mu(F\circ r_\varphi)=\nu_q(F)$ for every $F\in C(S_\varphi(Ax_0))$. The entire fiber 
$$\operatorname{Lift}(\nu_q)=\{\mu\in\mathcal P(S_1(Ax_0)):(r_\varphi)_*\mu=\nu_q\}$$
is a nonempty compact convex subset of $\mathcal P(S_1(Ax_0))$. Thus the following diagrams commute:
\begin{center}
\begin{tikzcd}[column sep=large,row sep=large]
	\mathcal P\bigl(S_1(Ax_0)\bigr)
	\arrow{r}{(r_\varphi)_*}
	&
	\mathcal P\bigl(S_\varphi(Ax_0)\bigr)
	\\
	\{q\}
	\arrow[mapsto]{u}{\mu_q}
	\arrow[equals]{r}
	&
	\{q\}
	\arrow[mapsto]{u}[swap]{\nu_q}
\end{tikzcd}
\end{center}
\begin{center}
\begin{tikzcd}[column sep=huge,row sep=large]
	C\bigl(S_\varphi(Ax_0)\bigr)
	\arrow{r}{r_\varphi^*:F\mapsto F\circ r_\varphi}
	\arrow{d}[swap]{\nu_q}
	&
	C\bigl(S_1(Ax_0)\bigr)
	\arrow{d}{\mu_q}
	\\
	\mathbb R
	\arrow[equals]{r}
	&
	\mathbb R
\end{tikzcd}
\end{center}
\end{corollary}

\begin{proof}
The restriction map $r_\varphi:S_1(Ax_0)\to S_\varphi(Ax_0)$ is a continuous surjection between compact Hausdorff spaces. Its pullback is therefore an isometric unital embedding from $C(S_\varphi(Ax_0))$ into $C(S_1(Ax_0))$. On its range define $\ell_0(F\circ r_\varphi)=\nu_q(F)$. The functional $\ell_0$ has norm one and sends $1$ to $1$.

The Hahn--Banach theorem extends $\ell_0$ to a norm-one functional $\ell$ on $C(S_1(Ax_0))$ with $\ell(1)=1$. Such a functional is positive: if $0\le G\le1$, then $\ell(G)=1-\ell(1-G)\ge1-\|1-G\|_\infty\ge0$. By the Riesz representation theorem, $\ell$ is integration against a regular Borel probability measure $\mu$. Its definition gives $\mu(F\circ r_\varphi)=\nu_q(F)$, equivalently $(r_\varphi)_*\mu=\nu_q$.

Finally, $\mathcal P(S_1(Ax_0))$ is weak-$*$ compact and the pushforward map $(r_\varphi)_*$ is continuous and affine. Hence $\operatorname{Lift}(\nu_q)$ is the nonempty inverse image of the closed singleton $\{\nu_q\}$ and is compact and convex. 
\end{proof}

\section{Local stability and harmonic barycenters}
\label{sec:local-stability-harmonic-barycenters}

Throughout this section, we fix a regular Greenian open set $U$, a point
$x_0\in U$, and a dense set $A\subseteq U$. Write $Ax_0=A\cup\{x_0\}$.
Thus the Martin compactification and every harmonic type considered below are local objects attached to $(U,x_0)$.

\subsection{Evaluation types and local stability}
\label{subsec:evaluation-types-local-stability}

Let $T$ be the complete theory of the chosen local $\mathcal L_{\mathrm{Har}}$-structure, together with any explicitly named predicates used in this section.

\begin{definition}[Evaluation type]
For $h\in H_1^+(U,x_0)$, the evaluation type over $Ax_0$ is
$\operatorname{tp}_{\operatorname{Ev}}(h/Ax_0)$, namely the complete
record of all predicates obtained from $\operatorname{Ev}(h;a)=\operatorname{Eval}_U(a,h)=h(a)$,
with $a\in Ax_0$, by continuous connectives. A partial evaluation
type records only some of these conditions.
\end{definition}

An evaluation type must be distinguished from the minimal types in Section~4. There, a point of $U$ is the variable and the harmonic data are parameters. In an evaluation type, the harmonic function is the variable and the points of $Ax_0$ are parameters.

\begin{definition}[Finite satisfiability]
\label{def:finite-satisfiability-evaluation-types}
A partial evaluation type $\pi(h)$ over $Ax_0$ is finitely
satisfiable in $H_1^+(U,x_0)$ if, for every finite
$\pi_0\subseteq\pi$ and every $\varepsilon>0$, there is
$h_{\pi_0,\varepsilon}\in H_1^+(U,x_0)$ such that
$\varphi(h_{\pi_0,\varepsilon};\bar a)\le\varepsilon$ for every
condition $\varphi(h;\bar a)=0$ in $\pi_0$.
\end{definition}

\begin{theorem}[Harnack realization of evaluation types]
\label{thm:harnack-realization-evaluation-types}
Every partial evaluation type over $Ax_0$ that is finitely
satisfiable in $H_1^+(U,x_0)$ is realized by a member of
$H_1^+(U,x_0)$.
\end{theorem}

\begin{proof}
Let $\mathcal I$ be the directed set of pairs
$(\pi_0,\varepsilon)$, where $\pi_0\subseteq\pi$ is finite and
$\varepsilon>0$. Declare
$(\pi_0,\varepsilon)\preceq(\pi_1,\delta)$ when
$\pi_0\subseteq\pi_1$ and $\delta\le\varepsilon$. By finite
satisfiability, choose $h_{\pi_0,\varepsilon}\in
H_1^+(U,x_0)$ for each index so that every condition in $\pi_0$ has
value at most $\varepsilon$.

On each compact subset of $U$, Harnack's inequality
bounds the net uniformly because every member is positive harmonic
and has value one at $x_0$. The Harnack compactness theorem therefore
gives a subnet converging locally uniformly to a nonnegative harmonic
function $h_\pi$. Evaluation at $x_0$ is continuous for locally
uniform convergence, so $h_\pi(x_0)=1$ and
$h_\pi\in H_1^+(U,x_0)$.

Fix a condition $\varphi(h;\bar a)=0$ in $\pi$ and $\eta>0$.
Eventually the directed net contains this condition and has error at
most $\eta$. The same estimate holds along the convergent subnet.
Since $\varphi$ is a continuous combination of finitely many
evaluations, local uniform convergence gives
$\varphi(h_\pi;\bar a)\le\eta$. Letting $\eta$ tend to zero shows
that $h_\pi$ realizes $\pi$. This uses Harnack compactness as stated in \cite[Theorem~1]{LW1965axiomatic}.
\end{proof}

For the remainder of this section we make the additional local
assumption that $U$ has a countable compact exhaustion
$C_0\Subset C_1\Subset\cdots$, with $x_0\in C_0$ and
$\bigcup_{n<\omega}C_n=U$. For each $n$, let $c_n\ge1$ be a Harnack
constant for $C_n$ relative to $x_0$, so $h(x)\le c_nh(x_0)$ for
$x\in C_n$ and every positive harmonic $h$ on $U$. If
$h\in H_1^+(U,x_0)$, put $h_n=h|_{C_n}$ and define
$\theta_n(h;x)=h_n(x)/c_n$ for $x\in C_n$. Thus
$0\le\theta_n\le1$.

Pass to a fixed local predicate expansion $\mathcal L_{\mathrm{Har}}(U,x_0)$
of the $Ax_0$-definitional expansion which names the compact loci $C_n$
as metric sorts and names each bounded uniformly continuous predicate
$\theta_n$ on $H_1^+(U,x_0)\times C_n$. %Before this expansion, $\theta_n$ is merely the displayed external function; after the expansion it is an ordinary predicate symbol. No definability of $C_n$ or $\theta_n$ in the original language is being asserted. 
Put
$\lambda_U=\max\{|\mathcal L|,|A|,\aleph_0\}$, so
$|\mathcal L_{\mathrm{Har}}(U,x_0)|\le\lambda_U$.
Recall the local stability theorem for an ordinary bounded
formula $\delta(u;v)$ in a fixed continuous language. For a structure $\mathfrak M=(M,\sigma_M)$ satisfying a theory $T$,
the local type space $S_\delta(M)$ carries the uniform metric
$d_\delta(p,q)=\sup_{b\in M}|\delta(p;b)-\delta(q;b)|$. We use
$\|\mathfrak M\|=\max\{\operatorname{dens}(M),|T|,\aleph_0\}$.

\begin{definition}[Good definition]
Let $p\in S_\delta(M)$. An $M$-definable predicate
$d_p\delta(v)$ is a definition of $p$ if
$d_p\delta(b)=\delta(p;b)$ for every $b\in M$. It is a good
definition if, for every elementary extension $\mathfrak N=(N,\sigma_N)\succeq\mathfrak M$, the
prescription $\delta(x;b)=d_p\delta(b)$, with $b\in N$, is a
consistent complete $\delta$-type over $N$ extending $p$. %Thus the	qualifier good means that the same definable predicate gives compatible extensions over all larger parameter sets.
\end{definition}

\begin{lemma}[Equivalent local stability criteria]
\label{lem:equivalent-local-stability-criteria}
For a bounded formula $\delta(u;v)$, the following conditions are
equivalent.
\begin{enumerate}
	\item The formula has no order property. Explicitly, there are no
	$\varepsilon>0$, $r\in\mathbb R$, and sequences $(a_i)_{i<\omega}$ and
	$(b_i)_{i<\omega}$ such that, whenever $i<j<\omega$, one has
	$\delta(a_i;b_j)\le r$ and
	$\delta(a_j;b_i)\ge r+\varepsilon$.
	
	\item For every infinite $\kappa\ge|T|$ and every $M\models T$ with
	$\operatorname{dens}(M)\le\kappa$, one has
	$\operatorname{dens}(S_\delta(M))\le\kappa$.
	
	\item The implication in (2) holds for some infinite cardinal $\kappa$ satisfying $\kappa^{|T|}=\kappa$.
	
	\item Every $p\in S_\delta(M)$ has a unique good definition
	$d_p\delta$.
	
	\item The double-limit criterion holds: for sequences $(a_i)_{i<\omega}$
	and $(b_j)_{j<\omega}$ for which both iterated limits exist, taking the
	limit first in $j$ and then in $i$ gives the same value as taking
	the limits in the reverse order.
	
	\item The family of instances is relatively weakly compact in $C(S_v(M))$.
\end{enumerate}
\end{lemma}

\begin{proof}
The equivalence of the order-property, type-density, and definability formulations is the continuous version of the local stability theorem \cite[Proposition~7.7]{Yaacov2008CONTINUOUSFO}. Grothendieck's
double-limit theorem identifies the double-limit and relative weak-compactness
conditions. Mazur's lemma is then applied after weak compactness, as explained in \cite[Theorem~5]{BenYaacov2014}.
\end{proof}

\begin{definition}[Canonical base and stationary type]
For a parameter set $B$, the definable closure
$\operatorname{dcl}^{\mathrm{eq}}(B)$ consists of the imaginaries
fixed by every automorphism fixing $B$, while the algebraic closure
$\operatorname{acl}^{\mathrm{eq}}(B)$ consists of the imaginaries
whose orbit over $B$ is compact. A set $C$ is algebraically closed
when $C=\operatorname{acl}^{\mathrm{eq}}(C)$.

For a stable formula $\delta$, the canonical base
$\operatorname{Cb}_\delta(p)$ is the imaginary canonical parameter
of the good definition $d_p\delta$. For a stable family $\Delta$,
put
$\operatorname{Cb}_\Delta(p)=
\operatorname{dcl}^{\mathrm{eq}}
\{\operatorname{Cb}_\delta(p):\delta\in\Delta\}$.
A $\Delta$-type over an algebraically closed set $C$ is stationary
if, for every $B\supseteq C$, it has exactly one extension whose good
definitions remain over $C$. When the family is clear, a realization
of this extension is denoted by $a\forkindep[C]B$.
\end{definition}

Now put $E=\operatorname{acl}^{\mathrm{eq}}(Ax_0)$.

\begin{theorem}[Stability of normalized harmonic evaluation]
\label{thm:theta-n-stable}
For every $n$, the $\mathcal L_{\mathrm{Har}}(U,x_0)$-formula $\theta_n(h;x)$ is stable.
\end{theorem}

\begin{proof}
The restrictions of $H_1^+(U,x_0)$ to $C_n$ form a relatively
compact subset of $C(C_n)$ in the uniform norm. Harnack's inequality
gives local boundedness and equicontinuity, while Harnack compactness
identifies every convergent subnet with the restriction of a member
of $H_1^+(U,x_0)$. Division by $c_n$ preserves relative
compactness. Relative compactness in norm implies relative weak
compactness, so condition~(6) of
Lemma~\ref{lem:equivalent-local-stability-criteria} applies.
\end{proof}

\begin{corollary}[Nonforking for $\theta_n$]
\label{cor:theta-n-nonforking}
Let $p=\operatorname{tp}_{\theta_n}(h/E)$, and let $q$ be an
extension of $p$ to $B\supseteq E$. Then $q$ is nonforking over $E$
if and only if its good definition is $E$-definable. Equivalently,
$\operatorname{Cb}_{\theta_n}(q)$ belongs to
$\operatorname{dcl}^{\mathrm{eq}}(E)$. This extension is unique.
\end{corollary}

This is the standard nonforking characterization for a stable formula \cite[Corollary~8.10]{Yaacov2008CONTINUOUSFO}.

\subsection{Gluing local stable formulas} 
\label{subsec:gluing-local-stable-formulas}

In this subsection, we apply the gluing technique in \cite{Yaacov2008CONTINUOUSFO} to extend the meaning of the
unique $\theta_n$-nonforking extension $h\forkindep[E]B$ described in Corollary~\ref{cor:theta-n-nonforking}.

\begin{lemma}[Finite continuous gluing]
\label{lem:finite-continuous-gluing}
For $i<m$, let $\delta_i(u;v)$ be a stable bounded formula on the
same loci. If $f:[0,1]^m\to[0,1]$ is continuous, then the formula obtained by applying $f$ pointwise to the tuple $(\delta_i)_{i<m}$ is stable.
\end{lemma}

This is \cite[Lemma 8.1]{Yaacov2008CONTINUOUSFO}.

To glue all compact levels, let
$Y_U=\{*\}\sqcup\bigsqcup_{n<\omega}(\{n\}\times C_n)$ be the
one-point compactification of the topological sum of $(\{n\}\times C_n)$, and
write $\iota_n:C_n\to Y_U$ for the inclusion of the $n$-th summand.
Extend $\mathcal L_{\mathrm{Har}}(U,x_0)$ to a language
$\mathcal L_{\mathrm{gl}}$ by naming the sort $Y_U$, the point $*$,
the maps $\iota_n$, and one predicate $\Theta_U$ whose prescribed
interpretation is
$$
\Theta_U(h;*)=0,\quad
\Theta_U(h;\iota_n(x))=2^{-n}\theta_n(h;x)
=\frac{2^{-n}}{c_n}h(x).
$$
Because $0\le\theta_n\le1$, the factors $2^{-n}$ make this predicate
continuous at $*$. Thus $\Theta_U$ is a genuine predicate symbol in the specified expansion $\mathcal L_{\mathrm{gl}}$.

Fix a countable collection $\mathcal C_0$ of rational piecewise-linear
continuous connectives that contains $\min$, $\max$, rational scalar
multiplication with truncation, and is uniformly dense in all finitary
continuous connectives. Let $\Delta_U$ consist of the formulas obtained
from the single atomic predicate $\Theta_U$ by finitely many
applications of connectives from $\mathcal C_0$, variable renaming,
interchange of object and parameter variables, and pullback along one
of the named maps $\iota_n$. The family is controlled in size:
$|\Delta_U|\le\lambda_U$. Moreover,
$\theta_n(h;x)=\min\{2^n\Theta_U(h;\iota_n(x)),1\}$, so the predicates $\theta_n$ need not be added separately.

\begin{theorem}[Glued $\Delta_U$-local independence calculus]
\label{thm:glued-local-independence-calculus}
In the complete local theory of the named $\mathcal L_{\mathrm{gl}}$-structure, every formula in $\Delta_U$ is stable. For
$E\subseteq C\subseteq B$, define $a\forkindep[C]B$ to mean that,
for every $\delta\in\Delta_U$, the good definition of the
$\delta$-fragment of $a$ over $B$ is definable over
$\operatorname{acl}^{\mathrm{eq}}(C)$. This
$\Delta_U$-local relation has the following properties.
\begin{enumerate}
	\item It is invariant under automorphisms, monotone in the right
	parameter set, base monotone, and has finite character.
	
	\item It is symmetric: $a\forkindep[C]b$ if and only if
	$b\forkindep[C]a$.
	
	\item It is transitive: if $C\subseteq B\subseteq D$, then
	$a\forkindep[C]D$ is equivalent to
	$a\forkindep[C]B$ together with $a\forkindep[B]D$.
	
	\item It has existence and extension: every
	$\Delta_U$-type over $C$ has an extension over $B$ whose
	realizations are independent from $B$ over $C$.
	
	\item It has local character. For every $a$ and $B\supseteq E$,
	there is $C\subseteq B$ containing $E$, with
	$\operatorname{dens}(C)\le\lambda_U$, such that
	$a\forkindep[C]B$.
	
	\item It is stationary over algebraically closed bases. If $C$
	is algebraically closed, $a$ and $a'$ have the same
	$\Delta_U$-type over $C$, and both
	$a\forkindep[C]B$ and $a'\forkindep[C]B$, then they have the same
	$\Delta_U$-type over $B$.
\end{enumerate}
\end{theorem}

\begin{proof}
We first verify stability of the language used in the reduction. If
$\Theta_U$ had an order pattern with gap $\varepsilon$, choose $N$
so that $2^{-N}<\varepsilon/4$. On every summand with $n>N$, the
oscillation of $\Theta_U$ is below $\varepsilon/4$, so two tail
parameters cannot witness the required gap. After discarding
finitely many indices, the order pattern would therefore be
supported on finitely many summands. On that finite union,
the summands are clopen and $\Theta_U$ is obtained by a finite
case distinction from the stable $\theta_n$'s. This is a finite
continuous gluing, contradicting
Lemma~\ref{lem:finite-continuous-gluing}. Hence $\Theta_U$ is stable,
and repeated application of that lemma shows that every formula in
$\Delta_U$ is stable.

All remaining assertions are local to $\Delta_U$. They do not define forking for the ambient Brelot theory. Their proof uses a finite stable reduction. Given a finite
subfamily $\Delta_0\subseteq\Delta_U$, Section~8.1 of
\cite{Yaacov2008CONTINUOUSFO} codes its finitely many fragments by
one continuous combination $\delta_{\Delta_0}$. The connective basis
contains the required finite maxima and rational rescalings, and
Lemma~\ref{lem:finite-continuous-gluing} makes
$\delta_{\Delta_0}$ stable. Proposition~8.7 of the same reference
then gives the unique definable extension of this finite stable
reduct over an algebraically closed base. The extensions obtained
from two finite families agree on their intersection, because both
are determined by the same good definitions. Compactness therefore
glues the compatible finite extensions to one
$\Delta_U$-type. This proves existence and extension. Equality of
the good definitions also proves stationarity.

Invariance, monotonicity, and finite character follow directly from
the definition and the finite reduction. Transitivity follows because
the definition over $C$ is unchanged first from $C$ to $B$ and then
from $B$ to $D$. For symmetry, apply symmetry for the single stable
formula $\delta_{\Delta_0}$ to every finite $\Delta_0$. The closure of
$\Delta_U$ under interchange of variables ensures that both
directions remain in the family.

Finally, for each $\delta\in\Delta_U$, local character for the stable
formula $\delta$ supplies a subset of $B$ of density at most
$\lambda_U$ over which its good definition is based. The union over
the at most $\lambda_U$ formulas in $\Delta_U$, followed by algebraic
closure, still has density at most $\lambda_U$. This gives the set
$C$ in item~(5). These are precisely the local arguments underlying
\cite[Lemma~8.11]{Yaacov2008CONTINUOUSFO}. The finite stable
reduction above is what permits their use here. 
\end{proof}

From now on, we use $\forkindep$ to denote the $\Delta_U$-local relation.

\subsection{Canonical bases and harmonic barycenters}
\label{subsec:canonical-bases-harmonic-barycenters}

For a variable tuple $v$, let $\operatorname{Def}_E(v)$ denote the
space of $E$-definable predicates in $v$, equipped with the uniform
norm. This space is closed under finite convex combinations and uniform
limits.

\begin{definition}[Stable barycenter]
Let $\delta(u;v)\in\Delta_U$ and let $\mu$ be a regular Borel
probability measure on $S_\delta(E)$. A $\delta$-barycenter of $\mu$
is an $E$-definable predicate $b_\mu^\delta\in
\operatorname{Def}_E(v)$ satisfying
$b_\mu^\delta(v)=\int d_p\delta(v)\,\mathrm{d}\mu(p)$ at every parameter
$v$. For a measure on $S_{\Delta_U}(E)$, its stable barycenter is the
family of these predicates as $\delta$ ranges over $\Delta_U$.
\end{definition}

\begin{theorem}[Definability and lift invariance of stable barycenters]
\label{thm:definability-lift-invariance-stable-barycenters}
Every regular Borel probability measure $\mu$ on
$S_\delta(E)$ has a unique $\delta$-barycenter. If
$r_\delta:S_1(E)\to S_\delta(E)$ is restriction, then the barycenter
of a complete measure $\rho$ depends only on
$(r_\delta)_*\rho$. In particular, complete lifts of the same local
measure have the same barycenter. The assertion is summarized by the commutative diagram
\begin{center}
	\begin{tikzcd}[column sep=large,row sep=large]
		\mathcal P(S_1(E))\arrow{r}{(r_\delta)_*}
		\arrow{d}[swap]{\operatorname{bar}_\delta\circ(r_\delta)_*}
		&\mathcal P(S_\delta(E))\arrow{d}{\operatorname{bar}_\delta}\\
		\operatorname{Def}_E(v)\arrow[equals]{r}
		&\operatorname{Def}_E(v).
	\end{tikzcd}
\end{center}
\end{theorem}

\begin{proof}
Consider the definition map $j_\delta:p\mapsto d_p\delta$ from $S_\delta(E)$ into $C(S_v(E))$. The local stability theorem does more than make the point evaluations continuous: it identifies $j_\delta$ with a weakly continuous map into a weakly compact subset $K_\delta$ of $C(S_v(E))$ \cite[Proposition~7.7]{Yaacov2008CONTINUOUSFO}. This weak continuity, rather than measurability of point evaluations alone, ensures that the Pettis integral $\int j_\delta(p)\,\mathrm{d}\mu(p)$ exists in the weakly closed convex hull of $K_\delta$. Its value at $v$ is exactly $\int d_p\delta(v)\,\mathrm{d}\mu(p)$.

Mazur's theorem now gives the required uniform approximation. For
every $\varepsilon>0$, there are
$p_1,\ldots,p_m\in S_\delta(E)$ and
$t_1,\ldots,t_m\in[0,1]$, with $\sum_it_i=1$, such that
$\|\sum_it_id_{p_i}\delta-b_\mu^\delta\|_\infty<\varepsilon$.
Each such finite convex combination is an $E$-definable predicate, since real continuous connectives are part of the continuous-logic syntax. %No perturbation of the coefficients is needed.
Since definable predicates are closed under uniform limits,
$b_\mu^\delta$ is $E$-definable.

The integral determines the value of the barycenter at every
parameter, so uniqueness is immediate. Finally, integration against a complete measure $\rho$ factors through its pushforward $(r_\delta)_*\rho$. This proves lift invariance and the commutativity of the diagram.
\end{proof}

\begin{definition}[Barycentric canonical base]
Let $\mu$ be a local measure on $S_{\Delta_U}(E)$. Its barycentric
canonical base is
$\operatorname{Cb}^{\mathrm{bar}}_{\Delta_U}(\mu)=
\operatorname{dcl}^{\mathrm{eq}}
\{\operatorname{cp}(b_\mu^\delta):\delta\in\Delta_U\}$,
where $\operatorname{cp}(b_\mu^\delta)$ denotes the canonical
parameter of the definable predicate $b_\mu^\delta$. %Thus an	infinite family is coded by the definable closure of its individual canonical parameters, not by an assumed single canonical parameter.
\end{definition}

\begin{corollary}[Nonforking preserves the barycentric canonical base]
\label{cor:nonforking-preserves-barycentric-cb}
Let $B\supseteq E$, and let
$\sigma_B:S_{\Delta_U}(E)\to S_{\Delta_U}(B)$ send a stationary local
type to its unique extension whose realizations $a$ satisfy
$a\forkindep[E]B$. The map $\sigma_B$ is Borel. Hence, for a local
measure $\mu$, the pushforward $\mu_B=(\sigma_B)_*\mu$ is defined and
satisfies
$b_{\mu_B}=b_\mu$ and
$\operatorname{Cb}^{\mathrm{bar}}_{\Delta_U}(\mu_B)=
\operatorname{Cb}^{\mathrm{bar}}_{\Delta_U}(\mu)$.
\end{corollary}

\begin{proof}
For $\delta\in\Delta_U$ and $b\in B$, the $b$-coordinate of
$\sigma_B(p)$ is $d_p\delta(b)$. The map
$p\mapsto d_p\delta$ is continuous for the weak topology, so each
such coordinate is Borel. The logic topology on
$S_{\Delta_U}(B)$ is generated by these coordinates. Consequently, $\sigma_B$ is Borel.

Nonforking extension leaves every good definition over $E$
unchanged. Thus the integrands defining $b_{\mu_B}^\delta$ and
$b_\mu^\delta$ agree for each $\delta$, and so do their integrals.
The equality of barycentric canonical bases follows from the
definable-closure definition above.
\end{proof}

Recall from Theorem~\ref{thm:second-Martin-isomorphism} that a
normalized harmonic type $q\in\mathscr H_1(Ax_0)$ has a representing
measure $\nu_q$ on the minimal Martin type locus. Define the Borel map
$\mathfrak h_\Theta(p)=\operatorname{tp}_{\Theta_U}(k_p/E)$ and put
$\lambda_q=(\mathfrak h_\Theta)_*\nu_q$.

\begin{corollary}[Recovery of the harmonic barycenter]
\label{cor:harmonic-barycenter-recovery}
The $\Theta_U$-barycenter of $\lambda_q$ recovers $q$. For every
$a\in A\cap C_n$, one has
$$
b_{\lambda_q}^{\Theta_U}(\iota_n(a))
=\frac{2^{-n}}{c_n}\int k_p(a)\,\mathrm{d}\nu_q(p)
=\frac{2^{-n}}{c_n}\operatorname{Ev}^q(a).
$$
Consequently
$\operatorname{Ev}^q(a)=2^nc_n
b_{\lambda_q}^{\Theta_U}(\iota_n(a))$. These values on the dense set
$A$, together with the normalization at $x_0$, determine the unique
member of $H_1^+(U,x_0)$ represented by $q$. Any two complete
Keisler measure lifts of $\lambda_q$ give the same recovered
harmonic barycenter.
\end{corollary}

\begin{proof}
The displayed identity follows from the definition of the stable
barycenter for the linear evaluation predicate $\Theta_U$ and from
the representing identity in
Theorem~\ref{thm:second-Martin-isomorphism}. Since the compact sets
$C_n$ exhaust $U$, every point of $A$ belongs to some $C_n$, and the
coefficient $2^{-n}/c_n$ is nonzero. The barycenter therefore
recovers the evaluation profile on $A$. Harmonic functions are
continuous and $A$ is dense, so this profile determines the function
on $U$. Lift independence is a direct application of
Theorem~\ref{thm:definability-lift-invariance-stable-barycenters}.
%Thus the recovery statement uses only the stable local predicate $\Theta_U$ and is not an application of global stability of the	ambient theory.
\end{proof}

\section{First jets and fine potential theory}

\subsection{The expanded jet language}

The abstract Brelot language does not by itself contain a derivative.
Consequently, differential rigidity should only be discussed after fixing a
smooth analytic model. Throughout this subsection, $U$ is either a relatively
compact regular domain in $\mathbb R^n$ or a regular coordinate domain in a
smooth $n$-manifold. The scalar harmonic sort is still $H(U)$, while a harmonic
map $F:U\to\mathbb R^m$ is an $m$-tuple from $H(U)$. This localization keeps the
jet construction separate from Brelot spaces that have no differentiable
structure.

This subsection is an interpretation in a fixed smooth structure, not an axiomatization of smooth harmonic spaces up to elementary equivalence. In particular, the chart maps, cotangent bundle, derivatives of transition maps, and interior derivative bounds are named data. Without them, neither $\mathrm{d}_xh$ nor $\det(\mathrm{d}_xF)$ is a formula of $\mathcal L_{\mathrm{Har}}$.

\begin{definition}[Language Jet]
\label{def:first-jet-language}
For every smooth regular $U$, let $J^1_U=J^1(U,\mathbb R)$ be a gauged
sort for scalar first jets. It comes with the base, value, and covector
projections $\pi_U:J^1_U\to X_U$, $\pi_0:J^1_U\to\mathbb R$, and
$\pi_1:J^1_U\to T^*U$. We add the function symbol
$j^1_U:X_U\times H(U)\to J^1_U$, whose intended interpretation is
$j^1_U(x,h)=(x,h(x),\mathrm{d}_xh)$. If $V\subseteq U$, the language also contains the
jet restriction $r^1_{V,U}:J^1_U|_{X_V}\to J^1_V$. The resulting language is
denoted by $\mathcal L_{\mathrm{Jet}}$.

For an $m$-tuple $F=(h_1,\ldots,h_m)$, the notation $j^1_U(x,F)$ means the
corresponding product jet. When $m=n$, the truncated Jacobian predicate is
$\operatorname{Jac}_U(x,F)=\min\{1,|\det(\mathrm{d}_xF)|\}$. Thus
$\operatorname{Jac}_U(x,F)=0$ says exactly that the first jet of $F$ is
singular at $x$.
\end{definition}

The symbols in Definition~\ref{def:first-jet-language} are required to satisfy
the following axioms. They are an axiom scheme because the continuity bounds
are stated on each bounded harmonic ball and each compact coordinate chart.

\begin{definition}[First-jet scheme]
\label{def:first-jet-axioms}
The chosen smooth jet expansion satisfies the following requirements on every named compact chart and bounded harmonic locus.
\begin{enumerate}
	\item The jet lies over its evaluation point, so $\pi_U(j^1_U(x,h))=x$.
	Its value component is the old evaluation predicate:
	$\pi_0(j^1_U(x,h))=\operatorname{Eval}_U(x,h)$.
	
	\item The covector part is linear in the harmonic variable. In shorter
	notation, $D_x(ah+bk)=aD_xh+bD_xk$ for named scalars $a,b$.
	
	\item Jets are natural under restriction. For every $V\subseteq U$, the
	following diagram commutes:
	\begin{center}
		\begin{tikzcd}[column sep=large,row sep=large]
			X_V\times H(U) \arrow{r}{\operatorname{id}\times\rho_{V,U}}
			\arrow{d}[swap]{j^1_U|_{X_V}}
			& X_V\times H(V) \arrow{d}{j^1_V} \\
			J^1_U|_{X_V} \arrow{r}[swap]{r^1_{V,U}}
			& J^1_V.
		\end{tikzcd}
	\end{center}
	The maps $r^1_{V,U}$ satisfy the same identity and cocycle laws as the
	restriction maps $\rho_{V,U}$.
	
	\item The covector is the differential of the named evaluation predicate.
	More explicitly, if $K$ is contained in a coordinate chart and $h$ ranges in
	a fixed bounded harmonic ball, there is a named modulus $\omega_K$ with
	$\omega_K(t)\to0$ as $t\to0$. Write
	$R_h(x,y)=h(y)-h(x)-\langle\mathrm{d}_xh,\chi(y)-\chi(x)\rangle$. The axiom requires
	$|R_h(x,y)|\le d_U(x,y)\omega_K(d_U(x,y))$ for $x,y\in K$. Under a change of
	chart, the covector transforms by the usual transpose inverse of the
	derivative of the transition map.
	
	\item The jet symbols respect harmonic gluing. If local harmonic sections
	agree on overlaps and glue to $h$, their first jets agree under the maps
	$r^1_{V,U}$ and glue to $j^1_U(-,h)$. All jet projections, restrictions, and
	determinant predicates have the uniform continuity moduli supplied by the
	interior derivative estimates on the chosen bounded sorts.
\end{enumerate}
\end{definition}

The Taylor axiom prevents an arbitrary covector field from being attached to
a harmonic function. It also shows that, once the smooth harmonic reduct is
fixed, the interpretation of $j^1_U$ is forced. The other important
compatibility is with the Dirichlet operator. For $x\in U$ and vector-valued
boundary data, the triangle
\begin{center}
\begin{tikzcd}[column sep=large,row sep=large]
	J_x^1(U,\mathbb R^m)  
	& H(U)^m  \arrow[swap]{l}{j^1_{U,x}} \\
	B(U)^m \arrow{u}{j^1_{U,x}\circ D_U^m} \arrow[equals]{r}
	& B(U)^m \arrow[swap]{u}{D_U^m}
\end{tikzcd}
\end{center}
commutes. Here $j^1_{U,x}$ means evaluation of the jet at the fixed point $x$.
This diagram says that solving the Dirichlet problem and then taking the first
jet gives the same definable germ as the composite symbol named on the
diagonal.

For a named definable compact $K\Subset U$ and a harmonic map
$F:U\to\mathbb R^n$, let
$\Sigma_K(z/F)$ be the singular-jet partial type consisting of the condition
$z\in K$ together with $\operatorname{Jac}_U(z,F)=0$. The point of using a
partial type is that non-vanishing is an omission statement. If the Jacobian
is positive on $K$, compactness gives a number $\delta_K>0$ such that
$|\det(\mathrm{d}_xF)|\ge\delta_K$ on $K$. %no uniform lower bound near the boundary is being asserted.

\begin{theorem}[RKC and Lewy in the jet language]
\label{thm:RKC-Lewy-jet}
Let $\Omega\subseteq\mathbb R^2$ be a Jordan domain and let $Q\subseteq
\mathbb R^2$ be convex. If an orientation-preserving boundary homeomorphism
$f:\partial\Omega\to\partial Q$ has harmonic Dirichlet extension
$F=D_\Omega^2(f)$, then $F$ is a diffeomorphism from $\Omega$ onto $Q$ and
$\Sigma_K(z/F)$ is omitted for every $K\Subset\Omega$. This is the
Rad\'o--Kneser--Choquet phenomenon in first-jet form
\cite{IwaniecOnninen2014RKC,IwaniecOnninen2019RKC}.

More generally, a degree-one harmonic map between planar domains equipped
with the smooth metrics considered by Martin has no singular first jet. Thus
its singular-jet partial type is omitted and the map is a diffeomorphism
\cite{Martin2016CurvedLewy}.
\end{theorem}

\begin{proof}
The RKC theorem first gives injectivity of the harmonic extension into the
convex target. Planar Lewy rigidity then says that the Jacobian of this
harmonic homeomorphism never vanishes. In the language above, this is exactly
the omission of $\Sigma_K(z/F)$ on every compact $K$. Martin's curved-metric
theorem replaces convex boundary data by the degree-one hypothesis and proves
the same pointwise invertibility of $\mathrm{d}_xF$. The inverse function theorem then
turns the jet statement into the asserted local, and hence global,
diffeomorphism statement.
\end{proof}

The $p$-harmonic version in \cite{IwaniecOnninen2019RKC} is read in the same
way after adjoining a $p$-Dirichlet solution sort and its first-jet symbol. It
is not literally a sentence about the linear sort $H(U)$ when $p\ne2$. The
common content is that suitable topological boundary data force the
singular-jet type to be omitted.

\subsection{Thin, polar, and finely open sets}

Fix $U\in\tau_0$. For $n\ge1$, use the bounded predicate
$\operatorname{Ev}_n(x,u)=\min\{u(x),n\}/n$ on $X_U\times S^+(U)$. The family
$\{\operatorname{Ev}_n:n\ge1\}$ determines the extended value of $u$. In
particular, $u(x)=+\infty$ if and only if
$\operatorname{Ev}_n(x,u)=1$ for every $n$. We assume that the language has a
countable regular basis and definable restricted infima over the admissible
sets under consideration.

\begin{definition}[Barrier partial type]
\label{def:barrier-partial-type}
Let $A\subseteq U$ be admissible and let $x\in U$. A barrier code is a tuple
$\lambda=(n,r,s,V)$, where $n\ge1$, $r,s\in\mathbb Q\cap[0,1]$, $r<s$, and
$V$ is a basic neighborhood of $x$. The partial type
$\Pi^\lambda_{A,x}(u)$ says
$\operatorname{Ev}_n(x,u)\le r$ and
$\inf_{y\in(A\setminus\{x\})\cap V}\operatorname{Ev}_n(y,u)\ge s$. If the
punctured infimum is not primitive, these conditions are written over a
countable exhaustion by closed annuli around $x$.
\end{definition}

\begin{lemma}[Barrier characterization]
\label{lem:barrier-characterization-thinness}
The set $A$ is thin at $x$ if and only if
$\Pi^\lambda_{A,x}$ is realized for some barrier code $\lambda$.
\end{lemma}

\begin{proof}
The classical superharmonic barrier criterion for thinness gives a positive
superharmonic $u$ such that $u(x)$ is strictly smaller than the lower limit of
$u$ along $A\setminus\{x\}$ approaching $x$ \cite{doob1984classical}. After
shrinking $V$, truncating at a sufficiently large $n$, and choosing rational
numbers between the two values, one obtains a realization of
$\Pi^\lambda_{A,x}$. Conversely, a realization of this partial type supplies
the required strict superharmonic separation, and hence a barrier at $x$.
\end{proof}

There is no reason for $\Pi^\lambda_{A,x}$ to have a unique complete
extension. If $T_{A,x}=\operatorname{Th}_{\mathcal L(A,x)}(\mathfrak M)$,
then the correct object is the closed set
$[\Pi^\lambda_{A,x}]=\{q\in S_{S^+(U)}(T_{A,x}):q\supseteq
\Pi^\lambda_{A,x}\}$. A model realizes the partial type precisely when it
realizes at least one member of this set. Consequently, omitting one selected
completion does not omit the barrier type. Every completion must be omitted.

Call a partial type uniformly principal if a single formula has a nonempty
zero set in every model of the complete theory and every element of that zero
set realizes the partial type. This is the continuous-logic form of
principality needed below.

\begin{theorem}[Omitting partial types and thinness]
\label{thm:OTT-persistent-thinness}
Assume that $T_{A,x}$ and its language are countable. Then $A$ is thin at $x$
in every model of $T_{A,x}$ if and only if
$\Pi^\lambda_{A,x}$ is uniformly principal for some barrier code $\lambda$.
If every barrier partial type is non-principal, there is a separable model of
$T_{A,x}$ omitting all of them, and $A$ is not thin at $x$ in that model.
\end{theorem}

\begin{proof}
There are only countably many barrier codes. If one partial type is uniformly
principal, it is realized in every model, so
Lemma~\ref{lem:barrier-characterization-thinness} gives thinness in every
model. Conversely, if none is principal, the omitting type theorem applied
directly to the countable family of partial types gives a separable model
omitting all of them \cite[Theorem~4.9]{Eagle2014omiting}. The barrier
characterization then shows that thinness fails there. This argument neither
chooses nor assumes a unique complete extension.
\end{proof}

Since $W$ is finely open at $x\in W$ exactly when $U\setminus W$ is thin at
$x$, Theorem~\ref{thm:OTT-persistent-thinness} immediately yields the
following formulation.

\begin{corollary}[A sufficient condition for thinness]
\label{cor:OTT-fine-openness}
If $\Pi^\lambda_{A,x}$ is uniformly principal for some barrier code $\lambda$, then $A$ is thin at $x$ in every model of $T_{A,x}$.
\end{corollary}

Now suppose that $U$ admits a strictly positive finite continuous potential
$p_U$. The standard reduction and balayage criteria can be stated entirely in
the expanded language $\mathcal L_{\mathrm{Bal}}$.

\begin{theorem}[Balayage criteria]
\label{thm:balayage-polar-thin-fine}
Let $A,W\subseteq U$ be admissible.
\begin{enumerate}
	\item The set $A$ is polar if and only if
	$\operatorname{Bal}_A(p_U)=0$.
	
	\item The set $A$ is thin at $x$ if and only if some relatively compact
	regular neighborhood $V$ of $x$ satisfies
	$\operatorname{Bal}_{(A\setminus\{x\})\cap V}(p_U)(x)<p_U(x)$.
	
	\item The set $W$ is finely open if and only if, for every $x\in W$, some
	relatively compact regular neighborhood $V$ of $x$ satisfies
	$\operatorname{Bal}_{(U\setminus W)\cap V}(p_U)(x)<p_U(x)$.
\end{enumerate}
\end{theorem}

\begin{proof}
The first assertion is the polarity theorem in balayage form, while the
second is the classical reduction criterion for thinness
\cite{doob1984classical}. The puncture is essential when $x\in A$, because
thinness tests the remaining germ of $A$ at $x$. The third assertion follows
from the second by applying it to the complement of $W$.
\end{proof}

To obtain a genuinely measure-theoretic statement, return to the fixed Greenian domain $(U,x_0)$ and assume the hypotheses of Theorem~\ref{thm:second-Martin-isomorphism}. Let $h>0$ be harmonic, normalize it at $x_0$, and let $\nu_h$ be the local Keisler measure corresponding to its canonical Martin measure. If $w>0$ is a potential, define the boundary-pole locus
$$P_\infty(w/h)=\left\{p\in\partial_m^{\mathrm{np}}(Ax_0):\operatorname*{mf-\limsup}_{x\to p}\frac{w(x)}{h(x)}=+\infty\right\},$$
where $\operatorname{mf\!\!\lim}$ denotes the minimal-fine limit. When minimal-fine approach filters are included among the admissible codes, this is the realization set of the corresponding boundary-pole partial type.

\begin{corollary}[Boundary poles are Keisler measure zero]
\label{cor:Keisler-polar-characterization}
For every positive potential $w$ and positive harmonic function $h$ as above, $\nu_h(P_\infty(w/h))=0$. More strongly, $w/h$ has minimal-fine limit zero at $\nu_h$-almost every minimal Martin type.
\end{corollary}

\begin{proof}
The Fatou--Na\"im--Doob theorem says that $w/h$ has minimal-fine limit zero at almost every point of the minimal Martin boundary with respect to the canonical representing measure of $h$ \cite[Theorem 8]{Gowrisankaran1966FND}. Theorem~\ref{thm:second-Martin-isomorphism} transports that measure along the homeomorphism from minimal boundary points to minimal Martin types. Null sets, and in particular the boundary poles, remain null under this transport.
\end{proof}

Thus thinness is recorded by barrier realizations, fine openness by the same construction on a complement, and genuine boundary poles of potentials are null for the Keisler measure associated with a positive harmonic function. The last assertion uses classical minimal-fine boundary theory and is not a consequence of the zero function having the zero representing measure.

\subsection{Onshuus rank and obstruction chains}
\label{subsec:Onshuus-rank-obstruction-chains}

Fix an o-minimal expansion $\mathcal R$ of the ordered real field and put
$T_{\mathcal R}=\operatorname{Th}(\mathcal R)$. The coordinate analysis below is classical
first-order analysis in $T_{\mathcal R}$.

Recall that $\phi(x,b)$ {strongly divides} over $B$ if
$\operatorname{tp}(b/B)$ is nonalgebraic and
the family $\{\phi(x,b'):b'\equiv_Bb\}$ is $k$-inconsistent for some $k<\omega$.
Thorn-dividing, thorn-forking, and the corresponding thorn-$U$-rank
$U^{\mathrm b}$ are as in \cite[Definitions~2.1 and~4.1]{Onshuus2006Thorn}.
For a partial type $\Pi$, we put
$U^{\mathrm b}(\Pi)=\sup\{U^{\mathrm b}(p):p\in[\Pi]\}$, with value
$-\infty$ when $\Pi$ is inconsistent. We write
$a\forkindep[B]^{\mathrm b}C$ for thorn-independence, distinguishing it
from the $\Delta_U$-local relation of
Theorem~\ref{thm:glued-local-independence-calculus}.

\begin{proposition}[Thorn-rank in o-minimal theories]
\label{prop:Onshuus-rank-equals-ominimal-dimension}
For $p=\operatorname{tp}(a/B)$ in an o-minimal theory, $U^{\mathrm b}(p)=\dim(a/B)$, and, for $C\supseteq B$, $a\forkindep[B]^{\mathrm b}C$
if and only if $\dim(a/BC)=\dim(a/B)$.
Consequently, a thorn-forking chain starting at a type of dimension $r$
has height at most $r$, and this bound is attained.
\end{proposition}

\begin{proof}
This is the o-minimal computation of thorn-independence and thorn-$U$-rank
in \cite[\S5.2]{Onshuus2006Thorn}. In particular, thorn-forking is exactly
strict dimension drop. Successively adjoining an $\operatorname{acl}$-basis
of $a$ over $B$ realizes all drops from $r$ to $0$.
\end{proof}

\begin{example}[Wood's singular-jet chain]
\label{ex:Wood-ominimal-singular-jet}
Wood's semialgebraic harmonic homeomorphism $F_W(x,y,z)=(x^3-3xz^2+yz,y-3xz,z)$
\cite[p.~166]{Wood1991LewyFails} satisfies $\det(\mathrm{d}F_W)=3x^2$.
Hence its singular locus is the plane $P=\{x=0\}$. For $U=B(0,2)$ and $K=[-1,1]^3\Subset U$, the realization set of the
singular-jet partial type $\Sigma_K(w/F_W)$ is $P\cap K$, and therefore $U^{\mathrm b}(\Sigma_K)=2$.
Indeed, if $a=(0,b,c)$ is generic in $P\cap K$ over the parameters $B_0$
defining $F_W$ and $K$, put $B_1=B_0b$ and $B_2=B_1c$. Then
$\dim(a/B_0)=2$, $\dim(a/B_1)=1$, and $\dim(a/B_2)=0$,
and the equalities $w_2=b$ and $w_3=c$ strongly divide at the respective
stages. Thus $U^{\mathrm b}(\operatorname{tp}(a/B_i))=2-i$ for $i=0,1,2$.
This contrasts with Theorem~\ref{thm:RKC-Lewy-jet} since under its planar hypotheses the corresponding singular-jet partial type is inconsistent,
whereas Wood's three-dimensional obstruction has thorn-rank two.
\end{example}

\begin{example}[The Lebesgue-thorn germ]
\label{ex:Lebesgue-Rexp-thorn}
Now take the o-minimal field $\mathcal R=\mathbb R_{\exp}$, and let
$\rho(t)=\exp(-1/t)$ for $t>0$.  Set $T=\{0\}\cup
\{(t,y,z):0<t\le1/2,\ y^2+z^2\le\rho(t)^2\}$
and $\Omega=B(0,1)\setminus T$.
The Wiener--Gardiner criterion gives
$\int_0^{1/2}[t\log(t/\rho(t))]^{-1}\mathrm{d}t
=\int_0^{1/2}(1+t\log t)^{-1}\mathrm{d}t<\infty$.
Hence $T$ is thin at $0$ and $0$ is irregular for $\Omega$
\cite[Theorem~1 and p.~239]{Gardiner1992FineLimit}.  By
Lemma~\ref{lem:barrier-characterization-thinness}, the associated barrier
partial type is realized.

The non-principal germ at the tip is encoded by $\Gamma_T(w)=
\{0<w_1<1/m:m\ge2\}
\cup
\{w_2^2+w_3^2\le\rho(w_1)^2\}.$
Every finite subset is realized in $\mathbb R_{\exp}$, whereas the whole
type is omitted there and is consistent by compactness.

Choose a completion represented in an elementary extension by
$$a=(\epsilon,\rho(\epsilon)u,\rho(\epsilon)v),$$
where $\epsilon$ is positive infinitesimal and $(u,v)$ is generic in the open unit disc over
$\mathbb R\epsilon$. With
$B_0=\mathbb R$, $B_1=B_0\epsilon$, $B_2=B_1u$, and $B_3=B_2v$, the
definable change of variables
$(t,u,v)\mapsto(t,\rho(t)u,\rho(t)v)$ gives $\dim(a/B_i)=3-i$ for $i=0,1,2,3$.
The successive height, normalized first transverse coordinate, and normalized
second transverse coordinate strongly divide at the three stages. Hence $U^{\mathrm b}(\operatorname{tp}(a/B_i))=3-i$ for $i=0,1,2,3$,
and therefore $U^{\mathrm b}(\Gamma_T)=3$.

The completions of $\Gamma_T$ need not have the same rank. The axial
completion $(\epsilon,0,0)$ has rank one. A generic completion of the form
$(\epsilon,\rho(\epsilon)u,0)$ has rank two and the generic transverse
completion above has rank three. Thus $\{U^{\mathrm b}(p):p\in[\Gamma_T]\}=\{1,2,3\}$.
The Wiener calculation detects thinness, whereas thorn-rank records the
o-minimal geometry of its non-principal coordinate germs. %no equivalence between the two notions is asserted.
\end{example}

Thus the two examples have different o-minimal geometries: Wood's
interior obstruction is carried by a codimension-one singular locus,
whereas the Lebesgue boundary germ has completions of varying
thorn-rank.

\section*{Final remarks}

Several foundational questions remain open. A first difficulty is to find a complete theory of harmonic structures. The broad classes considered here are necessarily incomplete, whereas fixing the complete theory of a single analytic model may conceal geometric properties shared by its elementary substructures and extensions. Any stability analysis of o-minimal expansions also requires additional care. The final section confines its rank calculations to subsets of $\mathbb R^n$. More generally, the convex geometry of normalized positive harmonic types suggests a rank theory adapted to locally convex metric structures. Ideally, such a rank would assign minimal rank to extreme rays, reflect Choquet decomposition, and admit comparison with $U$-rank, dividing, and forking. A related problem concerns canonical bases. Martin profiles naturally encode canonical information associated with harmonic types, but their limits may live only in infinitary or measure-valued sorts. These codes should therefore not be identified automatically with imaginaries in the classical expansion $M^{\mathrm{eq}}$. Determining when these codes can be eliminated would be interesting. Finally, differential expansions of the harmonic language may connect this framework with classical real and complex analysis. By naming higher derivatives, one may seek model-theoretic formulations of further rigidity phenomena in harmonic-function theory. These questions remain open.

\bibliographystyle{alpha}
\bibliography{bibtex}
\end{document}